\documentclass[11pt, reqno]{amsart}
\usepackage{amsmath, amssymb}
\usepackage{graphicx,epsfig}
\usepackage{rawfonts}

\begin{document}
\parskip =\medskipamount
\newtheorem{theorem}{Theorem}[section]
\newtheorem{proposition}[theorem]{Proposition}
\newtheorem{lemma}[theorem]{Lemma}
\newtheorem{corollary}[theorem]{Corollary}
\newtheorem{defn}[theorem]{Definition}
\newtheorem{conjecture}[theorem]{Conjecture}

\theoremstyle{definition}
\newtheorem{definition}{Definition}

\numberwithin{equation}{section}

\newcommand{\R}{\mathbb R}
\newcommand{\TT}{\mathbb T}
\newcommand{\Z}{\mathbb Z}
\newcommand{\HH}{\mathbb H}
\newcommand{\vE}{\mathcal E} 
\newcommand{\dist}{\operatorname{dist}}
\newcommand{\supp}{\operatorname{supp}}
\newcommand{\diam}{\operatorname{diam}}
\parskip=\medskipamount

\newcommand{\sgn}{\operatorname{sign}}

\def\bfx{{x}}
\def\ve{\varepsilon}
\def\vp{\varphi}
\def\dist{\text {dist}}
\def\iiint{{\small{_\not}}\kern-3.5pt\int}
\def\mod{\text{\rm mod }}
\def\be{\begin{equation}}
\def\ee{\end{equation}}

\input pictexwd


\title[Restriction of toral eigenfunctions to hypersurfaces]
{Restriction of toral eigenfunctions to hypersurfaces and nodal
sets}
\author{Jean Bourgain and Ze\'ev Rudnick}
\address {School of Mathematics, Institute for Advanced Study,
Princeton, NJ 08540 } \email{bourgain@ias.edu}
\address {Raymond and Beverly Sackler School of Mathematical Sciences,
Tel Aviv University, Tel Aviv 69978, Israel}
\email{rudnick@post.tau.ac.il}
\date{\today}
\maketitle

\section{Introduction}

Let $M$ be a smooth Riemannian surface without boundary, $\Delta$
the corresponding Laplace-Beltrami operator and  $\Sigma$ a smooth
curve in $M$. Burq, G\'erard and Tzvetkov \cite{BGT} established
bounds for the $L^2$-norm of the restriction of eigenfunctions of
$\Delta$ to the curve $\Sigma$, showing that if \hfill\break
$-\Delta \varphi_\lambda = \lambda^2 \varphi_\lambda$, $\lambda>0$,
then
\begin{equation}\label{eq1}
 ||\varphi_\lambda||_{L^2(\Sigma)} \ll \lambda^{1/4} ||\varphi_\lambda||_{L^2(M)}
\end{equation}
and if $\Sigma$ has non-vanishing geodesic curvature then
\eqref {eq1} 
may be  improved to
\begin{equation}\label{eq2}
 ||\varphi_\lambda||_{L^2(\Sigma)} \ll \lambda^{1/6} ||\varphi_\lambda||_{L^2(M)}
\end{equation}
Both \eqref{eq1}, \eqref {eq2} are saturated for the sphere $S^2$.

In \cite{BGT} it is observed that for the flat torus $M=\TT^2$,
\eqref {eq1} 
can be improved to
\begin{equation}\label{eq3}
 ||\varphi_\lambda||_{L^2(\Sigma)} \ll \lambda^{\epsilon}
 ||\varphi_\lambda||_{L^2(M)},\quad \forall \epsilon>0
\end{equation}
due to the fact that there is a corresponding bound on the supremum
of the eigenfunctions.
They raise the question whether in \eqref {eq3} 
the factor $\lambda^{\epsilon} $  can be replaced by a constant,
that is whether there is a uniform $L^2$ restriction bound. As
pointed out by Sarnak \cite{Sar2}, if we take $\Sigma$ to be a
geodesic segment on the torus, 
this particular problem is essentially equivalent to the currently
open question of whether on the circle $|x|=\lambda$, the number of
lattice points  on an arc of size $ \lambda^{1/2}$ admits a uniform
bound.

In \cite{BGT} results similar to \eqref {eq1} 
are also established in the higher dimensional case for restrictions
of eigenfunctions to smooth
submanifolds, in particular \eqref {eq1} 
holds for codimension-one submanifolds (hypersurfaces) and is sharp
for the sphere $S^{d-1}$.
Moreover \eqref {eq2} 
remains valid for hypersurfaces with positive curvature \cite{H}.

In this paper we pursue the improvements of \eqref {eq2} 
for the standard flat $d$-dimensional tori $\TT^d=\R^d/\Z^d$,
considering the restriction to (codimension-one) hypersurfaces
$\Sigma$ with non-vanishing curvature.
\medskip

\noindent
{\bf Main Theorem.} 
{\sl  Let $d=2,3$ and let $\Sigma\subset \TT^d$ be a real analytic
hypersurface
 with non-zero curvature.
There are constants $0<c  <C  <\infty$ and $\Lambda >0$, all
depending on $\Sigma$, so that all eigenfunctions $\varphi_\lambda$
of the Laplacian on $\TT^d$ with  $\lambda>\Lambda$ satisfy
\begin{equation}\label{eq4}
 c ||\varphi_\lambda||_2 \leq ||\varphi_\lambda||_{L^2(\Sigma)} \leq C
 ||\varphi_\lambda||_2
\end{equation}}

Observe that for the lower bound, the curvature assumption is
necessary, since the eigenfunctions $\varphi(x) = \sin(2\pi n_1 x_1)
$ all vanish on the hypersurface $x_1=0$. In fact this lower bound
implies that a curved hypersurface cannot be contained in the nodal
set of eigenfunctions with arbitrarily large eigenvalues.

It was shown in \cite{B-R1} that this last property of the nodal
sets of toral eigenfunctions hold in arbitrary dimension $d$. As we
point out in Section~\ref{sec:intersection},  the argument from
\cite{B-R1} implies in fact  a bound for the $d-2$ dimensional
Hausdorff measure of the intersection of nodal sets with a fixed
hypersurface $\Sigma$:

\begin{theorem}\label{Theorem3}
Let $\Sigma\subset\mathbb T^d$ be a real analytic hypersurface with
nowhere vanishing curvature. Then for $\lambda >\lambda_\Sigma$, the
nodal set $N$ of any eigenfunction $\vp_\lambda$ satisfies
\be\label{eq5} h_{d-2} (N\cap\Sigma)< c_\Sigma \lambda. \ee
\end{theorem}

For dimension $d=2$, this means an upper bound for the number of
intersection points of a fixed curve with the nodal lines.
Interestingly, using the Main  Theorem, one can show that
conversely:

\begin{theorem}\label{IntTheorem4}
Let $\Sigma \subset\mathbb T^2$ be a real analytic non-geodesic
curve. There is $\lambda_\Sigma$ such that for
$\lambda>\lambda_\Sigma$, the nodal set $N$ of any eigenfunction
$\vp_\lambda$ satisfies \be\label{eq6} \#(N\cap\Sigma)\gg
\lambda^{1-\ve} \text { for all } \ve>0 \ee
\end{theorem}
\noindent and for $d=3$, the following property

\begin{theorem}
Let $\Sigma\subset\mathbb T^3$ be as in the Main Theorem.  There is
$\lambda_\Sigma$ such that for $\lambda>\lambda_\Sigma$, the nodal
set $N$ of any eigenfunction $\vp_\lambda$ intersects $\Sigma$.
\end{theorem}

Returning to the results of \cite{BGT} for smooth Riemannian
surfaces, let us point out that there is a close connection between
estimates on $\Vert\vp_\lambda\Vert_{L^{2}(\Sigma)}$ with $\Sigma$ a
geodesic segment and bounds on the $L^4$-norm
$\Vert\vp_\lambda\Vert_{L^4(M)}$. Recall Sogge's general estimate
for the $L^p$-norm \cite{So1} \be\label{eq7}
\Vert\vp_\lambda\Vert_{L^p(M)} \leq
C\lambda^{\delta(p)}\Vert\vp_\lambda\Vert_{L^2(M)} \ee
 where
\be \delta(p)=\begin{cases} \frac 12(\frac 12-\frac 1p) \text { if }
2\leq p \leq 6\\ \frac 12 -\frac 2p \text { if } 6\leq p\leq
\infty.\end{cases} \ee The following inequalities were established
in \cite {B} \be\label{eq8} \Vert \vp_\lambda\Vert_{L^2(\Sigma)}
\leq C\lambda^{\frac 1{2p}} \Vert\vp_\lambda\Vert_{L^p(M)} \ee
 if  $\Sigma\subset M$  is a geodesic segment  and $p\geq 2$,
and conversely \be\label{eq9} \Vert\vp_\lambda\Vert_{L^4(M)}\ll
\lambda^{\frac 1{16}+\ve} \max_\Sigma
\Vert\vp_\lambda\Vert_{L^2(\Sigma)}^{\frac 14}. \ee where the
maximum is over all geodesic segments $\Sigma\subset M$ of unit
length. Hence \eqref {eq8}, \eqref{eq9} imply that improving upon
the restriction bound \eqref{eq1} is essentially equivalent with
convexity breaking for the $L^4$-norm (see also \cite{So2}). Of
course for $M=\mathbb T^2, \Vert\vp_\lambda \Vert_\infty\ll
\lambda^\ve$ and previous considerations are of no interest.
However, the example of an integrable torus $M$ constructed in \cite
{B2} provides a sequence of eigenfunctions $\vp_\lambda$ and a
geodesic segment $\Sigma\subset M$ such that \be\label{eq10}
\Vert\vp_\lambda\Vert_{L^6(M)} \sim\lambda^{\frac 16} \text { and }
\Vert \vp_\lambda\Vert_{L^2(\Sigma)} \sim \lambda^{\frac 14}. \ee
Thus this example saturates the inequality \eqref {eq8} for $p=6$
and also the \cite{BGT} bound \eqref {eq1} (providing a surface
quite different from the sphere).

The proof of the Main Theorem for $d=2$ is rather simple (compared
with $d=3$) and we describe it next, as an illustration of the
method and some of the arithmetic ingredients used, see
\cite{BR-CRAS}.

Denote by $\sigma$ the normalized arc-length measure on the curve
$\Sigma$. Using the method of stationary phase, one sees that if
$\Sigma$ has non-vanishing curvature then the Fourier transform
$\widehat\sigma$ decays as
\begin{equation}\label{eq11} 
 |\widehat\sigma(\xi) | \ll |\xi|^{-1/2}, \quad \xi \neq 0.
\end{equation}
Moreover   $|\widehat\sigma(\xi)|\leq \widehat\sigma(0)=1$ with
equality only for $\xi=0$, hence
\begin{equation}\label{eq12} 
 \sup_{0\neq \xi\in \Z^2} |\widehat\sigma(\xi)| \leq 1-\delta,
\end{equation}
for some $\delta=\delta_\Sigma>0$.

An eigenfunction of the Laplacian on $\TT^2$ is a trigonometric
polynomial of the form:
$$
\varphi(x)  = \sum_{|n|=\lambda} \widehat\varphi(n)e(n\cdot x)
$$
(where $e(z):=e^{2\pi i z}$), all of whose frequencies lie in the
set $\vE:=  \Z^2 \cap \lambda S^1  $. As is well known, in dimension
$d=2$, $\#\vE\ll \lambda^\epsilon$ for all $\epsilon>0$. Moreover,
by a result of Jarnik \cite{Jar}, any arc on $\lambda S^1$ of length
at most $c\lambda^{1/3}$ contains at most two lattice points
(Cilleruelo and Cordoba \cite{CC} showed that for any $\delta<\frac
12$, arcs of length $\lambda^{\delta}$ contain at most $M(\delta)$
lattice points and in \cite{CG} it is conjectured that this remains
true for any $\delta<1$). Hence we may partition,

\begin{equation}
 \vE = \bigcup_\alpha \vE_\alpha
\end{equation}
where $\#\vE_\alpha \leq 2$ and $\dist(\vE_\alpha, \vE_\beta)>c
\lambda^{1/3}$ for $\alpha\neq \beta$. Correspondingly we may write,

\begin{equation}
 \varphi = \sum_\alpha \varphi^\alpha,\quad
\varphi^\alpha(x)  =\sum_{n\in \vE_\alpha} \widehat\varphi(n)e(nx),
\end{equation}

so that $||\varphi||_2^2 = \sum_\alpha ||\varphi^\alpha||_2^2$  and

\begin{equation}
 \int_\Sigma |\varphi|^2 d\sigma = \sum_\alpha \sum_\beta \int_\Sigma
 \varphi^\alpha \overline{\varphi^\beta} d\sigma.
\end{equation}

Applying \eqref{eq11}  we see that  $\int_\Sigma \varphi^\alpha
\overline{\varphi^\beta} d\sigma \ll \lambda^{-1/6}$ if $\alpha\neq
\beta$ and because $\#\vE \ll \lambda^\epsilon$ the total sum of
these nondiagonal terms is bounded by $\lambda^{-1/6 +\epsilon}
||\varphi||_2^2$. It suffices then to show that the diagonal terms
satisfy
\begin{equation}\label{eq14} 
 \delta ||\phi^\alpha||_2^2 \leq \int_\Sigma |\phi^\alpha|^2 d\sigma
 \leq 2 ||\phi^\alpha||_2^2
\end{equation}
This is clear if $\vE_\alpha=\{n\}$ while if $\vE_\alpha=\{m,n\}$
then

\begin{equation}
\int_\Sigma |\phi^\alpha|^2 d\sigma   = |\widehat\varphi(m)|^2 +
|\widehat\varphi(n)|^2 + 2\mbox{Re
}\widehat\varphi(m)\overline{\widehat\varphi(n)}
\widehat\sigma(m-n),
\end{equation}
and then \eqref {eq14} 
follows from \eqref{eq12} 
This proves the Theorem 
for  $d=2$.

The proof of the Main Theorem 
for dimension $d=3$   is considerably more involved and occupies
Sections 2--9 of the paper. Arguing along the lines of the
two-dimensional case gives an upper bound of $\lambda^{\epsilon}$.
To get the uniform bound for $d=3$
we need to replace the upper bound \eqref {eq11} 
for the Fourier transform of the hypersurface measure by an
asymptotic expansion, and then exploit cancellation in the resulting
exponential sums over the sphere. A key ingredient there is
controlling the number of lattice points in spherical caps.

To state some relevant results, denote as before by $\vE=\Z^d\cap
\lambda S^{d-1}$ the set of lattice points on the sphere of radius
$\lambda$. 
We have $\#\vE \ll \lambda^{d-2+\epsilon}$. Let $F_d(\lambda, r)$ be
the maximal number of lattice points in the intersection of $\vE$
with a spherical cap of size $r>1$.
A higher-dimensional analogue of Jarnik's theorem 
implies that if $r\ll \lambda^{1/(d+1)}$ then all lattice points in
such  a cap are co-planar, hence $F_d(r,\lambda)\ll
r^{d-3+\epsilon}$ in that case, for any $\epsilon>0$. For larger
caps, we show:

\begin{proposition}

i) Let $d=3$. Then for any $\eta<\frac 1{15}$,
 \begin{equation} \label{eq16} 
 F_3(\lambda,r) \ll \lambda^\epsilon \left(   r \Big(\frac
 r\lambda\Big)^\eta +1 \right)
 \end{equation}
ii) Let $d=4$. Then
\begin{equation}\label{eq17}
 F_4(\lambda,r) \ll \lambda^\epsilon \left( \frac{r^3}\lambda + r^{3/2} \right)
\end{equation}
iii) For $d\geq 5$ we have
\begin{equation}\label{eq18}
 F_d(\lambda,r) \ll \lambda^\epsilon \left( \frac{r^{d-1}}\lambda + r^{d-3} \right)
\end{equation}
(the factor $\lambda^\epsilon$ is redundant for large $d$).
\end{proposition}

The term $r^{d-1}/\lambda$ concerns the equidistribution of $\vE$,
while the term $r^{d-3}$ measures deviations related to accumulation
in lower dimensional strata.

Only \eqref{eq16} $(d=3)$ is relevant for our purpose (Lemma 6.8 in
the paper, proved in Section 9) and \eqref {eq17}, \eqref {eq18} for
$d\geq 4$ (proven in Appendix~\ref{sec:appendix}) were included to
provide a more complete picture. We point out that the argument used
to obtain \eqref{eq16} is based on certain diophantine
considerations and dimension reduction, hence differs considerably
from the proof of \eqref{eq17}, \eqref {eq18} using standard
Hardy-Littlewood circle method and Kloosterman's refinement for
$d=4$.

The second result expresses a mean-equidistribution property of
$\vE$. Partition the sphere $\lambda S^{2}$ into sets $C_\alpha$ of
size $\lambda^{ 1/2}$, for instance by intersecting with cubes of
that size. Since $\#\vE\ll \lambda^{1+\epsilon}$, one may expect
that $\#C_\alpha\cap \vE\ll \lambda^\epsilon$.
We show (in joint work with P. Sarnak \cite{BRS}) that as a
consequence of ``Linnik's basic Lemma'', this holds in the mean
square:

\begin{lemma}\label{IntLemma5.6}
\begin{equation}\label{eq19}
 \sum_\alpha [\#(\vE\cap C_\alpha)]^{2} \ll \lambda^{1+\epsilon}, \quad \forall \epsilon>0.
\end{equation}
\end{lemma}

Finally, considering very large caps $r>\lambda^{1-\delta}$, there
is an estimate

\begin{lemma}\label{IntLemma5.4}
\be\label{eq20} \# (\mathcal E\cap C_r)\ll \Big(\frac
r\lambda\Big)^2 \lambda^{1+\ve} \text { for } r>\lambda^{1-\delta_0}
\ee $(\delta_0>0$ some absolute constant)
\end{lemma}
\noindent which is a consequence of Linnik's equidistribution
property (see \S\ref{sec:9.1}). While we make essential use of Lemma
\ref{IntLemma5.6} in our analysis, Lemma \ref{IntLemma5.4} will not
be needed, strictly speaking.

Let $1<r<\lambda$ and let $C$, $C'$ be spherical $r$-caps on
$\lambda S^2$ of mutual distance at least $10r$. Following the
argument for $d=2$, we need to bound exponential sums of the form
\begin{equation}\label{eq21} 
 \sum_{n\in C} \sum_{n'\in C'} \widehat \varphi(n) \overline{\widehat\varphi(n')} e(\psi (n-n')) , \qquad ||\varphi||_2=1
\end{equation}
where $\psi$ is the support function of the hyper-surface $\Sigma$,
which appears in the asymptotic expansion of the Fourier transform
of the surface measure on $\Sigma$, see Section~\ref{sec:2}. For
instance, in the case that $\Sigma=\{|x|=1\}$ is the unit sphere
then $h(\xi) = |\xi|$.

For $r<\lambda^{1-\epsilon}$ we simply estimate \eqref{eq21}
by $F_3(\lambda,r)$ (see \eqref{eq16}). 
When $\lambda^{1-\epsilon} <r<\lambda$ this bound does not suffice
and we need to exploit
cancellation in the sum \eqref{eq21}. 

\begin{lemma}\label{IntLemma5.11}
 There is $\delta >0$ so that \eqref{eq21} 
admits a bound of $\lambda^{1-\delta}$ for $\lambda\gg 1$.
\end{lemma}

This statement depends essentially on the equidistribution of $\vE$
in caps of size $\sqrt{\lambda}$,
as expressed in Lemma \ref{IntLemma5.6}.  

Using Taylor expansions of the function $\psi(x-y) $ with $x, y$
restricted to $S^2$ and suitable coordinate restrictions, Lemma
\ref{IntLemma5.11} is eventually reduced to the following
one-dimensional exponential sum estimate (proven in
Section~\ref{sec:5}):
\begin{lemma}\label{IntLemma3.3}
 Let $\beta\gg 1$ and $X,Y\subset [0,1]$ arbitrary discrete sets such that $|x-x'|, |y-y'|>\beta^{-1/2}$
for $x\neq x'\in X$ and $y\neq y'\in Y$. Then
\begin{equation}\label{eq22}
 \left| \sum_{x\in X}\sum_{y\in Y} e( \beta xy + \beta^{1/3}x^2y^2) \right| \ll \beta^{1-\kappa}
\end{equation}
for some $\kappa>0$.
\end{lemma}
Extending the Main Theorem to arbitrary dimension $d$ remains
unsettled at this point. We make the following
\begin{conjecture}\label{Conjecture}
Let $d\geq 2$ be arbitrary and $\Sigma\subset\mathbb T^d$ a real
analytic hypersurface. Then, for some constant $C_\Sigma$, all
eigenfunctions $\vp_\lambda$ of $\mathbb T^d$ satisfy
\be\label{eq23} \Vert\vp_\lambda\Vert_{L^2(\Sigma)} \leq C_\Sigma
\Vert\vp_\lambda\Vert_2. \ee If moreover $\Sigma$ has nowhere
vanishing curvature and $\lambda>\lambda_\Sigma$, for some
$c_\Sigma>0$, also \be\label{eq24} \Vert
\vp_\lambda\Vert_{L^2(\Sigma)} \geq c_\Sigma
\Vert\vp_\lambda\Vert_2. \ee
\end{conjecture}

It should be pointed out that in our proof of the Main Theorem for
$d= 2,3$, only distributional properties of $\mathcal E=\mathbb
Z^d\cap [|x|=\lambda]$ were exploited, but not the fact that
$\mathcal E$ actually consists of lattice points. In Section 11, we
give an example, for $d\geq 8$, of sets $\mathcal
S_\lambda\subset\lambda S^{d-1}$ satisfying the `ideal'
distributional property \be\label{eq25} |x-y|\gtrsim \lambda^{\frac
1{d-1}} \text { for $x\not = y$ in $\mathcal S_\lambda$} \ee and
such that the  Fourier restriction operator \be\label{eq26}
L^2(S^{d-1}, d\sigma)\longrightarrow \ell^2(\mathcal S_\lambda):
\mu\mapsto \hat\mu|_{\mathcal S_\lambda} \ee has unbounded norm for
$\lambda\to\infty$. This illustrates the difficulty for carrying out
our analysis in larger dimension.

As said earlier, even for $d=2$ and $\Sigma$ a straight line segment
in $\mathbb T^2$, \eqref{eq23} remains open and is roughly
equivalent with the arithmetic statement that the number of lattice
points on an arc of size $\sqrt\lambda$ on the circle $|x|=\lambda$
is bounded by an absolute constant. An easy argument in \cite{BR2}
shows that this last property is true for most $E=\lambda^2$ and in
fact the elements
 of $\{|x|^2 =E\}$ are at least $\gg \lambda^{1-\ve}$ separated, for
 all $\ve>0$.
In Section~\ref{sec:generic}, we establish the following

\begin{theorem}\label{Theorem6}
Let $\Sigma\subset\mathbb T^2$ be a smooth curve. Then for almost
all $E=\lambda^2$, there is a uniform restriction bound
\be\label{eq27} \Vert\vp_\lambda\Vert_{L^2(\Sigma)} \leq
C\Vert\vp_\lambda\Vert_2. \ee
\end{theorem}

In Section~\ref{sec:NaSod} we obtain  an analogue for $\mathbb T^d$,
$ d\geq 3$ of a theorem of Nazarov and Sodin \cite {N-S} on the
number of nodal domains.

\begin{theorem}\label{IntTheorem7}
Let $d\geq 3$ and $E=\lambda^2$ be sufficiently large. Then for a
`typical' element $\vp_\lambda$ of the eigenfunction space
$-\Delta\vp =E\vp$, the nodal set $N$ has $\sim \lambda^d$
components.
\end{theorem}

Recalling Courant's nodal domain theorem, the interest of Theorem
\ref {IntTheorem7} is the lower bound on the number of nodal
domains.

Almost all the subsequent analysis in the paper relates to $d=3$ and
$\mathbb T^3$-eigenfunctions. Let us stress again that the
arithmetic structure of the frequencies of the trigonometric
polynomials involved is essential here.

\noindent{\bf Acknowledgement}: The authors are indebted to
P.~Sarnak for many stimulating discussions on the material presented
in the paper. J.B. was supported in part by N.S.F. grant DMS
0808042. Z.R. was supported by the Oswald Veblen Fund during his
stay at the Institute for Advanced Study and by the Israel Science
Foundation (grant No. 1083/10).

\section{Lattice Points in Spherical Caps}\label{sec:lattice pts}

\subsection{Lattice points on spheres}\label{sec:9.1}


We recall what is known concerning the total number $\rho_d(R^2)$ of
lattice points on the sphere of radius $R$. Throughout we assume, as
we may, that $n:=R^2$ is an integer. We have a general upper bound
\be\label{eq8.3} \rho_d(R^2)\ll R^{d-2+\epsilon}, \quad \forall
\epsilon >0 \ee and in dimension $d\geq 5$ we in fact have both a
lower and upper bound of this strength: \be\label{eq8.4}
\rho_d(R^2)\approx R^{d-2}, \quad d\geq 5. \ee In smaller dimensions
both the lower and upper bound \eqref{eq8.3} need not hold. For
instance if $n=2^k$ is a power of 2 then $\rho_4(R^2)=24$ is
bounded. The situation in dimension $d=3$ is particularly delicate.
It is known that $\rho_3(n)>0$ if and only if $n:=R^2\not=
4^k(8m-1)$. There are {\em primitive} lattice points on the sphere
of radius $R=\sqrt{n}$ (that is $x=(x_1,x_2,x_3)$ with
$\gcd(x_1,x_2,x_3)=1$) if an only if $n\neq 0,4,7\bmod 8$.
Concerning the number $\rho_3(R^2)$ of lattice points, the upper
bound \eqref{eq8.3} is still valid, and if there are primitive
lattice points then there is a lower bound of $ \rho_3(R^2)\gg
R^{1-o(1)}$ but there are arbitrarily large $R$'s so that
\be\label{eq8.6} \rho_3(R^2)\gg R\log\log R. \ee

A fundamental result conjectured by Linnik (and proved by him
assuming the Generalized Riemann Hypothesis), that for $n\neq 0,4,7
\bmod 8$, the projections of these lattice points
to the unit sphere become uniformly distributed  on the unit sphere
as $n\to \infty$. This was proved unconditionally by Duke \cite{D,
Duke-SP} and Golubeva and Fomenko \cite{G-F}, following a
breakthrough by Iwaniec \cite{Iwaniec}.

\subsection{Lattice points in spherical caps: Statement of results}
Let $\overset {_\rightarrow}\zeta \in S^{d-1}$ be a unit vector,
$R\gg 1$, and $r=o(R)$. Consider the spherical cap $C=C(R\overset
{_\rightarrow}\zeta, r)$ which is the intersection of the sphere
$|\overset{_\rightarrow} x|=R$ with the ball of radius $\approx r$
around $R\overset{_\rightarrow}\zeta$. Set
$$
F_d(R, r)=\max_{\overset{_\rightarrow}\zeta \in S^{d-1}}\# \mathbb
Z^d\cap C(R\overset{_\rightarrow} \zeta, r)
$$
which is the maximal number of lattice points in a spherical cap of
size $r$ on the sphere $|\overset{_\rightarrow} x|=R$. We want to
give an upper bound for $F_d(R, r)$ in the case of dimension $d=3$.
The results which will be proven in this section are as follows:

i) For all $\epsilon>0$, \be\label{eq8.1} F_3(R, r)\ll R^\epsilon
\Big(1+\frac{r^2}{R^{1/2}}\Big). \ee This is an immediate
consequence of a Jarnik-type result on non-coplanar lattice points
in small caps. It is only useful for small caps, when $r\ll
R^{1/2}$.

For larger caps we shall show  the following bound:

ii) For any $\eta<\frac 1{15}$, \be\label{eq8.2} F_3(R, r)\ll
R^\epsilon \Big(1+r\Big( \frac r R\big)^\eta\Big). \ee It is natural
to conjecture that $F_3(R, r)\ll R^\epsilon\big(1+\frac{r^2}R\big)$
for $r< R^{1-\delta}$.

\subsection{Intersections with hyperplanes}
Let $\kappa_d(R)$ be the maximal number of lattice points in the
intersection of the sphere $|\overset{_\rightarrow}x|= R$ in
$\mathbb R^d$ and a hyperplane.

For dimension $d=2$,
$$
\kappa_2(R)\leq 2
$$
while in dimension $d=3$ we have \be\label{eq8.7} \kappa_3(R)\ll
R^\epsilon, \quad\forall \epsilon >0. \ee

\subsection{Small caps}

\begin{lemma}\label{Lemma8.8}
For a spherical cap $C$ of size $r$ on the sphere of radius $R$ in
$\mathbb R^3$ the number of lattice points in $C$ is at most
\be\label{eq8.9} \#C\cap\mathbb Z^3 \ll
R^\epsilon\Big(1+\frac{r^2}{R^{\frac 12}}\Big). \ee
\end{lemma}

\begin{proof}
Firstly, we note that if the cap has radius $r\ll R^{1/4}$ then it
contains only a $O(R^\epsilon)$ lattice points. This can be deduced
from Jarnik's method \cite {Jar} and also from a general result of
Andrews \cite{A} that if $C$ is any convex body in $\mathbb R^d$
with volume $V$ then the number of lattice points on its boundary
which are not coplanar is $\ll V^{\frac{d-1}{d+1}}$. In our case of
a cap in dimension 3, the base of the cap has area $\approx r^2$ and
if $\theta$ is the opening of the cap, so that $r\approx R\theta$,
then the height of the cap is about $R-R\cos\theta\approx R\theta^2
\approx r^2/R$, hence the volume of the cap is $V\approx r^4/R$.
Thus if $r<R^{1/4}$ then any such cap will contain at most (say) 100
non-coplanar lattice points. Any  lattice points in the cap will lie
on one of the plane sections of the cap through any three of the 100
non-coplanar lattice points. Each such plane section will contain at
most $R^\epsilon$ lattice points (uniformly as a function of the
plane) and hence the cap will contain at most $O(R^\epsilon)$
lattice points.

Now, for a cap $C$ of radius $r\geqq R^{1/4}$, divide it into caps
of radius $R^{1/4}$; the number of such caps will be $\approx$
area$(C)/(R^{1/4})^2\approx r^2/R^{1/2}$, and hence the total number
of lattice points in $C$ is at most $R^\epsilon(1+r^2/R^{1/2})$.
\end{proof}

\subsection{A linear and sub-linear bound}
We now turn to larger caps.

Here is a simple bound via slicing, using the fact that we can
control the number of lattice points in the intersection of a sphere
and a hyperplane parallel to one of the coordinate hyperplanes:

\begin{lemma}\label{Lemma8.10}
In dimension $d\geq 2$, \be \label{eq8.11} F_d(R, r)\ll
(1+r)\kappa_d(R). \ee
\end{lemma}

\noindent{\sl Proof.} A ball of radius $r$ is contained in a
vertical slab of the form \hfill\break $A<x_d<A+2r$ and hence all
integer points in the intersection of the sphere
$|\overset{_\rightarrow} x|=R$ and the ball $|
\overset{_\rightarrow} x -\overset{_\rightarrow}x_0|<r$ lie in the
union of the planes $x_d=k$, $A\leq k\leq A+2r$ with $k$ integer.
The intersection of each plane and the sphere
$|\overset{_\rightarrow} x|=R$ has at most $\kappa_d(R)$ lattice
points, and therefore the total number of lattice points is at most
$(1+r)\kappa_d(R)$.

In particular, for dimension $d=3$ this says that \be\label{eq8.12}
\#C\cap\mathbb Z^3\ll R^\epsilon (1+r). \ee

We can improve on Lemma~\ref{Lemma8.10} by slicing with well-chosen
planes rather than vertical planes. More precisely, we have

\begin{lemma}\label{Lemma8.13}
Let $C$ be a cap of size $r$ on the sphere
$\{|\overset{_\rightarrow} x| =R\}\subset\mathbb R^3$. Then for any
$0<\eta< 1/16$, \be\label{eq8.13} \#C\cap\mathbb Z^3 \ll R^\epsilon
\Big(1+r\Big(\frac rR\Big)^\eta\Big). \ee
\end{lemma}

\noindent{\sl Proof.}

It will involve several considerations.

i) {\bf Finding good slices.}  We try to find an integer vector
$\overset{_\rightarrow} a\in\mathbb Z^d$ and use slices of the cap
with the sections $\overset{_\rightarrow} a . \overset{_\rightarrow}
x =k$. We consider a larger cap $C_1=C(R{\frac
{\overset{_\rightarrow} a}{|\overset{_\rightarrow} a|}}, R\theta_1)$
of radius $r_1=R\theta_1$ around $R{\frac{\overset{_\rightarrow}
a}{|\overset{_\rightarrow} a|}}$ which contains the original cap
$C$. Thus we want the new cap angle $\theta_1$ to satisfy
\be\label{eq8.13'} \theta_1=\theta+\Big|\overset{_\rightarrow} \zeta
-\frac {\overset{_\rightarrow} a}{ |\overset{_\rightarrow} a|}\Big|.
\ee

To bound the number of lattice points in the new cap $C_1$, we
exhaust them by the parallel sections $\overset{_\rightarrow} a .
\overset{_\rightarrow}x=k$, which are orthogonal to the direction
$\overset{_\rightarrow} a$ of the new cap. The distance between
adjacent sections is $1/|\overset{_\rightarrow} a|$. The number of
sections intersecting the cap $C_1$ is bounded by
$|\overset{_\rightarrow} a|$ times the height of the cap, which is
$R-R\cos\theta_1\approx R\theta^2_1$. Hence the number $\nu(C_1,
\overset{_\rightarrow} a)$ of sections intersecting the cap is
\be\label{eq8.14} \nu(C_{1}, \overset{_\rightarrow} a)\ll
1+R\theta_1^{2} |\overset{_\rightarrow} a| \ee and the analysis
above shows that the number of lattice points in the cap $C_1$ is
bounded by \be\label{eq8.15} \#C_1 \cap\mathbb Z^d \leq
\kappa_d(R)\cdot\nu (C_1, \overset{_\rightarrow} a)\ll
\kappa_d(R)\cdot (1+R\theta^2_1|\overset{_\rightarrow} a|). \ee To
gain over the linear bound \eqref{eq8.12} we need to find some
$\delta>0$ and a nonzero integer vector $\overset{_\rightarrow}
a\in\mathbb Z^d$ such that \be\label{eq8.16}
R\theta_1^2|\overset{_\rightarrow}a|\ll r\theta^{2\delta} \ee that
is \be\label{eq8.17} \theta+|\overset{_\rightarrow}\zeta
-\frac{\overset{_\rightarrow}a}{|\overset{_\rightarrow} a|}| \ll
\frac {\theta^{\frac 12+\delta}}{|\overset{_\rightarrow} a|^{\frac
12}}. \ee Setting \be\label{eq8.18} Q=\theta^{-1+2\delta}=
\Big(\frac Rr\Big)^{1-2\delta} \ee then \eqref{eq8.17} is implied by
requiring both \be\label{eq8.19} |\overset{_\rightarrow} a|\leq Q
\ee \be \label{eq8.20} \Big|\overset{_\rightarrow}\zeta
-\frac{\overset{_\rightarrow} a}{|\overset{_\rightarrow}
a|}\Big|\leq \frac 1{|\overset{_\rightarrow} a| ^{\frac 12}Q^{\frac
12+\gamma}} \ee where we have set \be\label{eq8.21} \gamma:
=\frac{2\delta}{1-2\delta}. \ee Finding $\overset{_\rightarrow} a
\in \mathbb Z^d$ as above is then our goal.

\bigskip
(ii) {\bf Diophantine approximation}

\begin{lemma}\label{Lemma8.22}
Fix an integer $Q\geq 1$ and $0<\eta< 1$. Let
$\overset{_\rightarrow}\zeta =(\zeta_1, \ldots, \zeta_d)\in [-1,
1]^d$.  Then one of the following holds:

(1) There is $Q\leq q\leq 2Q$ and $\overset{_\rightarrow} a\in
\mathbb Z^d$ such that \be\label{eq8.23} \max_{1\leq j\leq d}
\Big|\zeta_j-\frac{a_j}{q}\Big|<\frac \eta Q. \ee

(2) There are $\overset{_\rightarrow} b \in\mathbb Z^d$, with
$0<\max|b_j|<1/\eta$ with \be\label{eq8.24} \Big\Vert\sum^d_{j=1}
b_j\zeta_j\Big\Vert<\frac c{Q\eta^d} \ee where we denote by $\Vert
x\Vert$ the fractional part of $x$, or the distance of $x$ to the
nearest integer.
\end{lemma}

\noindent {\sl Proof.} Let $0\leq \psi\leq 1$ be a smooth bump
function on the torus $\mathbb T^d$, such that

\begin{enumerate}
\item $0\leq\psi\leq 1$ for $\Vert x\Vert<\eta/2$
\item $\psi(x)=0$ for $\Vert x\Vert> \eta$
\item $|\hat\psi(m)|\ll \eta^d e^{-\sqrt{\eta|m|}}$
\end{enumerate}

If \eqref{eq8.23} fails, then
$$
\max_{Q\leq q< 2Q,  1\leq j\leq d}\Vert q\zeta_j\Vert\geq \eta
$$
hence \be\label{eq8.25} \sum_{Q\leq q<2Q} \psi
(q\overset{_\rightarrow} \zeta)=0. \ee Expressing this in a Fourier
series gives (writing $e(x):= e^{2\pi i x})$
$$
\begin{aligned}
0&=Q\hat\psi(0)+\sum_{0\not= b\in\mathbb Z^d} \hat\psi(b) \sum_{Q\leq q<2Q} e(q\zeta\cdot b)\\
&> cQ\eta^d-c\sum_{b\not=0} \eta^d e^{-\sqrt{\eta|b|}} \Big(|e(\zeta\cdot b)-1|+\frac 1Q\Big)^{-1}\\
&> cQ\eta^d\Big(1-c\sum_{b\not=0} e^{-\sqrt{\eta|b|}} \ \frac 1{Q\Vert\zeta\cdot b\Vert+1}\Big)\\
&> cQ\eta^d\Big(\frac 12-c\eta^{-d}\max_{0<|b|<c\eta^{-1}} \ \frac
1{Q\Vert\zeta\cdot b\Vert}\Big).
\end{aligned}
$$
Hence $Q\Vert\zeta\cdot b\Vert<c\eta^{-d}$ for some nonzero
$b\in\mathbb Z^d, |b|<c\eta^{-1}$.
\medskip

\begin{lemma}\label{Lemma8.26}
Let $\zeta_1, \zeta_2 \in [-1,1], 0<\gamma<1/15$, and $Q\gg _\gamma
1$ an integer.

Then there is an integer $1\leq q\leq Q$ and $\overset{_\rightarrow}
a \in\mathbb Z^2$ so that \be\label{eq8.27} \max_{j=1, 2}
\Big|\zeta_j - \frac{a_j}q\Big| \ll \frac 1{q^{\frac 12} Q^{\frac
12+\gamma}}. \ee
\end{lemma}

\noindent {\bf Remark.} Dirichlet's principle says that given
$\overset{_\rightarrow} \zeta \in\mathbb R^2$, and an integer $K\geq
1$, we can find $1\leq q\leq K^2$ and $\overset{_\rightarrow} a\in
\mathbb Z^2$ so that \be\label{eq8.28} \max_{j=1, 2}\Big|\zeta_j
-\frac {a_j}q\Big| <\frac 1{qK}. \ee Lemma \ref{Lemma8.26} improves
on this when $q$ is small.

\noindent {\sl Proof.} Applying Lemma \ref{Lemma8.22} with $\eta
=Q^{-\gamma}$, either we have an integer \hfill\break $Q\leq q \leq
2Q$ with
$$
\Big|\zeta_j-\frac{a_j}{q}\Big| < \frac 1{Q^{1+\gamma}}\leq\frac
{\sqrt 2}{q^{\frac 12}Q^{\frac 12+\gamma}}
$$
which gives us what we need, or else the second option in the
statement of the lemma occurs, that is there is some nonzero vector
$\overset{_\rightarrow} b\in\mathbb Z^2$ with $|b_1|\leq|b_2| \leq
Q^\gamma$, and $a\in\mathbb Z$ so that \be\label{eq8.29}
|b_1\zeta_1+ b_2\zeta_2 +a|<\frac 1{Q^{1-2\gamma}} \ee that is
\be\label{eq8.30} \Big|\zeta_2+\frac{b_1}{b_2} \zeta_1 +\frac
a{b_2}\Big| < \frac 1{|b_1|Q^{1-2\gamma}}. \ee Now choose an integer
$Q_1$ so that \be\label{eq8.31} 2Q^{\frac{1+3\gamma}2}<Q_1 <\frac 14
Q^{1-6\gamma} \ee which is possible if $0<\gamma<1/15$ and $Q\gg
_\gamma 1$.

Using Dirichlet's principle, there is some $1\leq q_1\leq Q_1$ and
an integer $a'\in\mathbb Z$ so that \be\label{eq8.32} \Big|\zeta_1
-\frac {a'}{q_1} \Big| <\frac 1{q_1Q_1}. \ee Define $a_1, a_2
\in\mathbb Z$ by \be\label{eq8.33} a_1=a'|b_2|, \quad -a_2=\pm (b_1
a'+aq_1), \quad q=q_1|b_2| \ee We claim that these satisfy the
statement of the Lemma. Indeed, by \eqref{eq8.32} we have
\be\label{eq8.34} \Big|\zeta_1
-\frac{a_1}{q}\Big|=\Big|\zeta_1-\frac{a'}{q_1}\Big|<\frac 1{q_1Q_1}
\ee and due to \eqref{eq8.31} we have, since $q_1=q/|b_2|\geq
qQ^{-\gamma}$, that \be\label{eq8.35} \frac 1{q_1Q_1}< \frac
1{2q_1Q^{\frac 12+\frac{3\gamma}2}} <\frac 1{2q_1^{\frac 12}
Q^{\frac 12+\frac{3\gamma} 2}}\leq \frac 1{2q^{\frac 12} Q^{\frac
12+\gamma}} \ee giving $|\zeta_1-\frac {a_1}q|<\frac 1{2q^{\frac
12}Q^{\frac 12+\gamma}}$. Moreover using the small linear relation
\eqref{eq8.30} between $\zeta_1$ and $\zeta_2$ and replacing
$\zeta_1$ by $a_1/q=a'/q_1$ we find
$$
\begin{aligned}
\Big|\zeta_2 - \frac {a_2}{q}\Big|&= \Big|\zeta_2+\frac{b_1}{b_2} \frac {a'}{q_1} +\frac a{b_2}\Big|\\
&\leq \Big|\zeta_2 +\frac {b_1}{b_2} \zeta_1 +\frac a{b_2} \Big|+\frac{|b_1|}{|b_2|} \Big|\zeta_1 -\frac {a'}{q_1}\Big|\\
&<\frac 1{|b_2|Q^{1-2\gamma}} +\frac 1{q_1 Q_1}
\end{aligned}
$$
Now since $q_1<Q_1<\frac 14 Q^{1-6\gamma}$ we have
$$
\frac 1{|b_2|Q^{1-2\gamma}}\leq \frac 1{|b_2|^{\frac 12} Q^{1-2\gamma}} = \frac{q_1^{\frac 12}}{q^{\frac 12} Q^{1-2\gamma}}\\
$$
\be \label{eq8.36} < \frac 1{2q^{\frac 12} Q^{\frac 12+\gamma}} \ee
and combining with \eqref{eq8.35} we get
$$
\Big|\zeta_2-\frac {a_2}q\Big| < \frac 1{q^{\frac 12} Q^{\frac
12+\gamma}}
$$
as claimed.
\medskip

\subsection{Proof of Lemma \ref{Lemma8.13}}

Assuming that $|\zeta_3| =\max |\zeta_j|$, we apply
Lemma~\ref{Lemma8.26} to $\big(\frac {\zeta_1}{\zeta_3},
\frac{\zeta_2}{\zeta_3})$ to find $1\leq q\leq Q$ and $a_1, a_2
\in\mathbb Z$ so that \be\label{eq8.37} \max_{j=1, 2}
\Big|\frac{\zeta_j}{\zeta_3} -\frac {a_j} q\Big|<\frac 1{q^{\frac
12} Q^{\frac 12 +\gamma}}. \ee Set $\overset{_\rightarrow} a=(a_1,
a_2, q)$ then $|\overset{_\rightarrow} a|\approx q$ and
\be\label{eq8.38} \Big|\overset{_\rightarrow} \zeta - \zeta_3\frac
1{q} \overset{_\rightarrow} a\Big|<\frac 1{q^{\frac 12} Q^{\frac
12+\gamma}} \approx \frac 1{|\overset{_\rightarrow} a|^{\frac 12}
Q^{\frac 12+\gamma}} \ee Since for any pair of nonzero vectors
$\overset{_\rightarrow} c$, $\overset{_\rightarrow} d$ we have by
the triangle inequality \be\label{eq8.39}
\Big|\frac{\overset{_\rightarrow} c}{|\overset{_\rightarrow} c|}
-\frac {\overset{_\rightarrow} d}{|\overset{_\rightarrow} d|}\Big|
\leq 2\frac {|\overset{_\rightarrow} c-\overset{_\rightarrow} d|}
{|\overset{_\rightarrow} c|} \ee and hence also
$$
\Big| \overset{_\rightarrow} \zeta -
\frac{\overset{_\rightarrow}a}{|\overset{_\rightarrow} a|}\Big|\ll
\frac 1{|\overset{_\rightarrow} a|^{\frac 12} Q^{\frac 12+\gamma}}.
$$
Thus we have found $\overset{_\rightarrow} a$ satisfying
\eqref{eq8.19}, \eqref {eq8.20}, completing the proof of
Lemma~\ref{Lemma8.13}. \qed

\section{The Fourier transform of  surface-carried measures}
\label{sec:2}

Let $\Sigma\subset\mathbb R^3$ be a real analytic surface with
non-vanishing curvature and $p\in\Sigma$. Applying a rigid motion,
we may assume $p=0$ and $\Sigma$ locally parametrized around $(0, 0,
0)$ by a map \be\label{eq1.1} (x_1, x_2) \mapsto \big(x_1, x_2, \phi
(x_1, x_2)\big) \ee where $\phi$ is real-analytic on a neighborhood
of $(0, 0)$ as has  the form \be\label{eq1.2} \phi(x_1, x_2)=
a_1x_1^2+ a_2 x_2^2 +\sum_{\alpha+\beta\geq 3} a_{\alpha\beta}\,
x_1^\alpha \, x_2^\beta \ee with
$$
a_1\not= 0, \quad a_2 \not= 0, \quad
|a_{\alpha\beta}|<C^{\alpha+\beta} \;.
$$
Distinguishing the case $a_1a_2>0$ (positive curvature) and
$a_1a_2<0$ (negative curvature), we need to consider the two models
\be\label{eq1.3} \phi (x_1, x_2)= x_1^2 +x_2^2
+\sum_{\alpha+\beta\geq 3} a_{\alpha \beta}\,  x_1^\alpha \,
x_2^\beta \ee and \be\label{eq1.4} \phi(x_1, x_2)= x_1^2-x_2^2
+\sum_{\alpha+\beta\geq 3} a_{\alpha\beta}\, x_1^\alpha \,
x_2^\beta. \ee

Denote by $\sigma$ the surface measure of $\Sigma$. Let
$\xi\in\mathbb R^3$ ($|\xi|$ large) and evaluate the Fourier
transform
$$
\int_{\Sigma\text { (local)}} e^{ix\xi} \sigma(dx)
$$
\be \label{eq1.5}
 =\int e^{i\big(x_1\xi_1 + x_2\xi_2+\phi(x_1, x_2)\xi_3\big)} \omega(x)\, dx_1 \, dx_2
\end{equation}
where $\omega$ is some smooth function supported by a (small)
neighborhood of $(0, 0)$.

The critical points of the phase function satisfy \be\label{eq1.6}
\begin{cases} \xi_1+\partial_1\phi(x)\xi_3 = \xi_1+ (2x_1+ \sum_{\alpha+\beta\geq 3} \alpha a_{\alpha\beta} x_1^{\alpha-1} x_2^\beta)\xi_3=0\\
\xi_2 +\partial_2 \phi(x) \xi_3 =\xi_2 + ( 2\epsilon x_2+
\sum_{\alpha+\beta\geq 3} \beta a_{\alpha\beta} x_1^\alpha
x_1^{\beta-1}) \xi_3= 0
\end{cases}
\ee where $\epsilon =\pm 1$ depending on whether we are in case
\eqref{eq1.3} or \eqref{eq1.4}.

It follows that in $\supp \omega$ there are no critical points
unless \be\label{eq1.7} |\xi_1|, |\xi_2|< c|\xi_3| \ee ($c$ a small
constant, depending on $\supp \omega$).

If \eqref{eq1.7} there is a unique critical point \be\label{eq1.8}
x=x(\xi)=\Big(x_1\Big(\frac{\xi_1}{\xi_3}, \frac{\xi_2}{\xi_3}\Big),
x_2 \Big(\frac{\xi_1}{\xi_3}, \frac{\xi_2}{\xi_3}\Big)\Big) \ee
where \be\label{eq1.9}
\begin{cases}
2x_1+\sum_{\alpha+\beta\geq 3} \alpha a_{\alpha\beta} x_1^{\alpha-1} x_2^\beta = -\frac{\xi_1}{\xi_3}\\
2\epsilon x_2 +\sum_{\alpha+\beta\geq 3} \beta a_{\alpha\beta}
x_1^\alpha x_2^{\beta-1} =-\frac{\xi_2}{\xi_3}\end{cases} \ee and
$|a_{\alpha\beta}|<C^{\alpha+\beta}$.

From the stationary phase formula (see \cite{G-S}, Ch. 1)
\be\label{eq1.10} \eqref{eq1.5} = \frac{2\pi}{|\xi_3|} e^{\frac{\pi
i}{4} \sgn H (x(\xi) ) } \ \frac{e^{i\psi(\xi)}} {|\det
  H\big(x(\xi)\big)|^{1/2}} \
\omega\big(x(\xi)\big) + 0\left (\frac 1{|\xi|^{2}}\right) \ee where
$H(x)$ is the Hessian of $\phi$ at $x$, $\sgn H$ is the signature of
$H$ and \be\label{eq1.11}
\begin{aligned}
\psi(\xi) & = x_1(\xi)\xi_1+ x_2 (\xi)\xi_2+ \phi\big(x_1(\xi), x_2(\xi)\big)\xi_3\\[6pt]
&\overset{\eqref{eq1.6}} =  -\xi_3 \left\{ x_1(\xi)\partial_1 \phi
\big( x(\xi)\big)+x_2 (\xi)\partial_2 \phi\big(x(\xi)\big)
-\phi(x_1(\xi), x_2(\xi))\right\}.
\end{aligned}
\ee

By \eqref{eq1.3}, \eqref{eq1.4} \be\label{eq1.12}
H(x)=\left[\begin{matrix} 2+\sum\limits_{\substack {\alpha+\beta\geq
3 \\ \alpha\geq 2}} \alpha(\alpha-1) \, a_{\alpha\beta} \,
x_1^{\alpha-2} \, x_2^\beta \ & \sum\limits
_{\substack{\alpha+\beta\geq 3\\ \alpha\geq 1, \beta\geq 1}}
\, \alpha\beta a_{\alpha\beta} \, x_1^{\alpha-1} \, x_2^{\beta-1}\\[18pt]
\sum\limits_{\substack {\alpha+\beta\geq 3\\ \alpha\geq 1, \beta\geq
1}} \alpha\beta \, a_{\alpha\beta} \, x_1^{\alpha-1} \,
x_2^{\beta-1} \ \ & \quad  2\epsilon+\sum\limits
_{\substack{\alpha+\beta\geq 3\\ \beta\geq 2}} \beta(\beta-1)\,
a_{\alpha \beta} \, x_1^\alpha \, x_2^{\beta-2}
\end{matrix}\right]
\ee and hence $\sgn H = 2 $ (resp. $0$) for positive (resp.
negative) curvature.

As will be clear soon, the error term $0(|\xi|^{-2})$ will be
harmless in our analysis in the restriction problem for $\mathbb
T^3$-eigenfunctions. The relevant contribution will be \be \frac
{e^{i\psi(\xi)}}{|\xi|} \ee coming from the main term. It turns out
that the decay factor $\frac 1{|\xi|}$ is barely insufficient to
ignore the oscillatory factor $e^{i\psi(\xi)}$. In order to exploit
this factor, a more careful analysis of the phase function
$\psi(\xi)$ is necessary.

Returning to \eqref{eq1.9}, \eqref{eq1.11}, we obtain by the
implicit function theorem \big(recalling \eqref{eq1.7}\big).
$$
\begin{cases}
x_1(\xi)= -\frac{\xi_1}{2\xi_3} +\sum_{\alpha+\beta\geq 2}\,  b_{\alpha\beta} \left(\frac{\xi_1}{\xi_3}\right)^\alpha \ \left(\frac {\xi_2}{\xi_3}\right)^\beta\\
x_2(\xi) =-\epsilon \frac{\xi_2}{2\xi_3} +\sum_{\alpha+\beta\geq 2}
c_{\alpha\beta} \left(\frac{\xi_1}{\xi_3}\right)^\alpha \
\left(\frac{\xi_2}{\xi_3}\right)^\beta
\end{cases}
$$
and \be\label {eq1.13} \psi(\xi) =-\frac 14
\left(\frac{\xi_1^2}{\xi_3}+\epsilon \frac{\xi_2^2}{\xi_3}\right)
+\sum_{\alpha+\beta\geq 3} d_{\alpha\beta} \, \xi_1^\alpha \,
\xi_2^\beta \, \xi_3^{1-\alpha-\beta} \ee $(|b_{\alpha, \beta}|,
|c_{\alpha\beta}|.|d_{\alpha\beta}| < C^{\alpha+\beta})$.

Thus $\psi(\xi)$ is homogeneous of degree one and hence
$\nabla\psi(\xi)$ is radially constant and $D^2\psi (\lambda\xi)
=\frac 1\lambda D^2 \psi (\xi)$. The self-adjoint matrix
$D^2\psi(\xi), \xi\not= 0$, has $\xi$ as eigenvector with eigenvalue
$0$.

From \eqref{eq1.13} \be \label{eq1.14} D^2\psi(\xi)= \begin{pmatrix}
-\frac 1{2\xi_3} &0&0\\ 0&-\frac\epsilon{2\xi_3} & 0\\
0&0&0\end{pmatrix} +0\Big(\frac{|\xi_1|+|\xi_2|}{\xi_3^2}\Big) \ee
and by \eqref{eq1.7}, we conclude that the other two eigenvalues of
$D^2\psi(\xi)$ are of size $\sim\frac 1{|\xi|}$ with same or
opposite sign depending on $\epsilon=1, -1$.

Hence \be\label{eq1.15} D^2 \psi(\xi)= \frac 1{|\xi|} P_{\xi^\bot} A
P_{\xi^\bot} \ee where $A$ is a self-adjoint operator (depending on
$\frac\xi{|\xi|}$), acting on $\xi^\bot$ and with eigenvalues
bounded from above and below (with same sign for $\epsilon=1$ and
opposite sign for $\epsilon=-1$).

\section{Spherical Restriction of the phase function}\label{sec:3}

Let $\psi(\xi)$ be the phase function obtained in
Section~\ref{sec:2} and
$$
S=S^2 =\{x\in\mathbb R^3, |x|=1\}.
$$
The domain of definition of $\psi$ is a cone $Z=\{|\xi_1|, |\xi_2|<
c|\xi_3|\}$, with $c>0$ a small constant, and $\psi$ is real
analytic on $Z$.

Subcones $Z'=\{|\xi_1|, |\xi_2|< c'|\xi_3|\}\subset Z, c'<c$, will
also be denoted by $Z$. We will need a normal form analysis of the
function $\psi(x-y)$ with $x, y$ restricted to $S$.

\begin{lemma}\label{Lemma2.1}
Let $p: O\overset{\text{open}}\subset \mathbb R^2\rightarrow C$ be a
real analytic parametrization of a cap $C\subset S$ such that $C\cap
(\xi+\bar Z)$ is connected for all $\xi\in C$.

Let $a, b\in O, a\not=b$ such that $p(a)-p(b)\in Z$. There are real
analytic coordinate changes $\alpha$ (resp. $\beta$) on a
neighborhood of $a$ (resp. $b$) such that \be\label{eq2.2}
\psi\big(p\big(a+\alpha(x)\big) - p(b+\beta(y)\big)\big)
=f(x)+g(y)+x_1y_1+x_2y_2+h(x, y) \ee with $f, g, h$ real analytic,
$h(x, y) =0(|x|^2 |y|^2)$ and $h\not= 0$.
\end{lemma}

\begin{proof}
(i) Letting $\eta= p(a)-p(b), \bar\eta =\frac \eta{|\eta|}$, it
  follows from \eqref{eq1.15} and curvature that the quadratic form
\[
D^2\psi(\eta)=\frac 1{|\eta|} P_{\eta^\bot} A_{\bar\eta}
P_{\eta^\bot}
\]
is non-degenerate on $\big(T_a -p(a)\big)\times \big(T_b-p(b)\big)$
where $T_a$ (resp. $T_b$) is the tangent space at $p(a) \in S $
(resp. $p(b)\in S$), as in figure~\ref{figure 1}.
\begin{figure}[h]
\begin{center}
\input fig1.tex
\end{center}
\caption{} \label{figure 1}
\end{figure}

Performing coordinate changes $\alpha, \beta$ in $x, y$ separately,
we can therefore obtain the form \eqref{eq2.2} with $h(x, y)=0(|x|^2
|y|^2)$. It remains to show that $h$ does not vanish identically.

(ii) Assume that $h=0$. Then \be\label{eq2.3} \psi
\big(p\big(a+\alpha(x)\big) - p\big(b+\beta(y)\big)\big) = f(x)+
g(y)+x_1y_1+x_2y_2 \ee for $x, y$ in a neighborhood of $(0, 0)\in
\mathbb R^2$.

Define $f_w(v)=\psi(v-w)$ where $w\in S$ is in a neighborhood $W$ of
$p(b)$ and $v\in S\cap (w+Z)$. It follows from \eqref{eq2.3} that
there is a neighborhood $V$ of $p(a)$ in $S$ (figure~\ref{figure 2})
such that \be\label{eq2.4} \dim [f_w|_V; w\in W]\leq 4. \ee Take
$\delta_0\gg \delta_1\gg \cdots\gg \delta_4$ and points $p(b) =w_0,
w_1, \ldots, w_4\in S$ in $W$ satisfying
$$
B(w_{i+1}, \delta_{i+1})\subset B(w_i, \delta_i)\cap (w_i+Z)
$$


\begin{figure}[h]
\begin{center}
\input fig2.tex
\end{center}
\caption{} \label{figure 2}
\end{figure}


From \eqref{eq2.4}, we may assume $f_{w_4}|_V$ a linear combination
of $f_{w_i}|_V$ ($0\leq i\leq 3$). Hence, invoking real analyticity,
it follows that $f_{w_4}$ is a linear combination of $f_{w_i} (0\leq
i\leq 3)$ on $\bigcap^4_{i=0} (w_i+Z)\cap S$. Since the functions
$f_{w_i}(0\leq i\leq 3)$ are smooth on $B(w_4, \delta_4)\cap
S\subset\bigcap^3_{i=0} (w_i+Z)$, it follows in particular that
$f_{w_4}$ is smooth on $B(w_{4}, \delta_4)\cap (w_4+\bar Z)\cap S$.
Hence, taking $u\in B(w_{4}, \delta_4)\cap (w_4+Z)\cap S$, it
follows that
$$D^2\psi(u-w_4)= \frac 1{|\zeta|} P_{\zeta^\bot}
A_{\frac\zeta{|\zeta|}} P_{\zeta^\bot}, \quad \zeta =u-w_4$$
restricted to $T_u-u$, is uniformly bounded for $u\in
 B(w_4, \delta_4)\cap (w_4+Z)\cap S$.

Thus for $\zeta$ as above and $\theta, \xi\in T_u-u$, $|\theta|=1 =
|\xi|$, \be\label{eq2.5} \langle AP_{\zeta^\bot}\theta, \xi\rangle
=0 (|\zeta|) \ee where $A=A_{\frac \zeta{|\zeta|}}$.

We show that this is not the case.

If $A$ is positive definite, take $\theta=\xi\in\zeta^\bot \cap
(T_u-u)$, $|\theta|=1$. Hence, from \eqref{eq2.5}, \be \label{eq2.6}
1\sim \langle A\theta, \theta\rangle = 0(|u-w_4|). \ee Letting $u\to
w_4$, we obtain a contradiction.

If $A$ is negative definite, proceed as follows.

Fix $\zeta =u-w_4$ and let $w_4'$ vary in $B(w_3, \delta_3)\cap S$
such that $u'=w_4'+\zeta\in B(w_3, \delta_3)\cap S$. Hence $w_4'$
varies over an arc of size $\sim\delta_3$ (see Figure~\ref{fig33}).
Let
$$\theta'\in  \zeta^\bot\cap (T_{u'} - u')=\zeta^\bot \cap
(T_{w_4'}-w_4'), \quad |\theta'|=1$$
\begin{figure}[h]
\begin{center}
\input fig33.tex
\end{center}
\caption{} \label{fig33}
\end{figure}

 From \eqref {eq2.5}
\be\label{eq2.7} \langle A\theta', \theta'\rangle =0(|\zeta|) \ee
where $A$ does not depend on $w_4'$ and $\theta'$ also varies over
an arc of size $\sim\delta_3$. Thus the left side of \eqref{eq2.7}
can be made at least $\sim \delta_3$, independently of $|\zeta|$, a
contradiction.

This proves Lemma \ref{Lemma2.1}.
\end{proof}

\begin{lemma}\label{Lemma2.8}
In the situation of Lemma \ref{Lemma2.1}, we may choose $a, b\in O$,
$p(a) -p(b) \in Z$ such that in \eqref{eq2.2} the function
\be\label{eq2.9}
\begin{aligned}
h(x, y)=& x_1^2 Q_{11} (y) +x_1x_2  Q_{12} (y)+x^2_2 Q_{22} (y)\\
&+0(|x|^3|y|^2 +|x|^2 |y|^3)
\end{aligned}
\ee where the $Q_{ij}(y)$ are quadratic forms, not all zero.
\end{lemma}

\begin{proof}
Start with \eqref{eq2.2} with $h(x, y)=0(|x|^2|y|^2), h\not= 0$.
Taking a sufficiently small $\delta>0$, it follows from the mean
value theorem that on $B(0, \delta)\times B(0, \delta)$, $B(0,
\delta)\subset \mathbb R^2$ \be\label{eq2.10} |\partial_x^2
\partial_y h|\leq \delta\big(\max_{B(0, \delta)}
|\partial^2_x\partial^2_y h|\big), \ee since $\partial^2_x
\partial_y h|_{y=0} =0$.

Similarly for $\partial_x\partial_y^2 h$.

It follows that there are $\bar x, \bar y\in B(0,\delta)\cap\mathbb
R^2$ such that \be \label{eq2.11} \Vert(\partial_x^2\theta_y h)(\bar
x, \bar y)\Vert+\Vert (\partial _x\partial_y^2 h)(\bar x, \bar
y)\Vert< \delta\Vert(\partial_x^2 \partial^2_yh) (\bar x, \bar
y)\Vert<1. \ee Setting $x=\bar x+\Delta x, y=\bar y+\Delta y$ in
\eqref{eq2.2}, we obtain after a linear coordinate change in $\Delta
y$

$$
\begin{aligned}
&\psi\big(p\big(a+\alpha(\bar x+\Delta x)\big) - p\big(b+\beta(\bar y+\Delta y)\big)\big)=\nonumber\\
&\bar f(\Delta x)+\bar g(\Delta y)+ (\Delta x)_1 (\Delta y)_1 +(\Delta x)_2(\Delta y)_2+\nonumber\\
\end{aligned}
$$
\be \label{eq2.12}
\sum_{i, j, k=1, 2} c_{ijk} (\partial_{x_ix_j}\partial_{y_k} h) (\bar x, \bar y)(\Delta x)_i (\Delta x)_j (\Delta y)_k+\\
\ee \be\label{eq2.13}
\sum_{i, j, k=1,2} c_{ijk} (\partial_{x_i} \partial_{y_j y_k} h)(\bar x, \bar y) (\Delta x)_i(\Delta y)_j(\Delta y)_k+\\
\ee \be\label{eq2.14}
\sum_{i, j, k, \ell=1, 2} c_{ijk\ell} (\partial_{x_i x_j} \partial_{y_k y_\ell} h)(\bar x, \bar y) (\Delta x)_i  (\Delta x)_j (\Delta y)_h (\Delta y)_\ell\\
\ee \be\label{eq2.15}
+ 0(|\Delta x|^3 |\Delta y|+ |\Delta x| \ |\Delta y|^3)\\
\ee
$$
+0(|\Delta x|^3 |\Delta y|^2+|\Delta x|^2 |\Delta y|^3)
$$
$a$ where \eqref{eq2.12}--\eqref{eq2.14} satisfy \eqref{eq2.11}.

We eliminate the $0(|\Delta x|^2|\Delta y|)$-terms in
\eqref{eq2.12}, \eqref{eq2.15} by  a coordinate change in $\Delta x$
and then the $0(|\Delta x| \ |\Delta y|^2)$-terms by a coordinate
change in $\Delta y$. Since the new quartic terms introduced by
these coordinate changes (in fact only the first) have coefficients
at most
$$
\begin{aligned}
&0\big(\Vert(\partial^2_x \partial_yh)(\bar x, \bar y)\Vert. \Vert(\partial_x\partial_y^2 h)(\bar x, \bar y)\Vert\big)\\
&< \delta \Vert\partial_x^2 \partial^2_y h(\bar x, \bar y)\Vert
\end{aligned}
$$
by \eqref{eq2.11}, the resulting expression will clearly still have
a nonvanishing bi-quadratic term. This proves Lemma \ref{Lemma2.8}.
\end{proof}

Denoting $F_{a, b}(x, y)=\psi\big(p\big(a+\alpha(x)\big) -
p(b+\beta(y)\big)\big)$ with $ h(x, y)$   in  \eqref{eq2.2}
satisfying Lemma \ref{Lemma2.8}, it follows that the Wronskian

\be\label{eq2.16} \max_{i, j, k, \ell=1, 2} W_{i, j, k, \ell}(F_{a,
b}) (0, 0)\equiv \max_{i, j, k, \ell =1, 2} \left|
\begin{matrix}
\partial_{x_1} \partial_{y_1} F &\partial_{x_1}\partial_{y_2} F& \partial_{x_1} \partial_{y_k y_\ell} F\\
\partial_{x_2} \partial_{y_1}F& \partial_{x_2} \partial_{y_2} F & \partial_{x_2} \partial_{y_ky_\ell} F\\
\partial_{x_ix_j}\partial_{y_1} F &\partial_{x_ix_j} \partial_{y_2} F &\partial_{x_ix_j}\partial_{y_ky_\ell} F\end{matrix}
\right| (0, 0) \not= 0. \ee

Note that property \eqref{eq2.16} does not depend on the
parametrization. Thus \be\label{eq2.17} \max W_{ijk\ell} (a, b)=\max
W_{i, j,k, \ell} \big(\psi\big(p(x)-p(y)\big)\big) (a, b) \not = 0.
\ee Invoking real analyticity, we obtain

\begin{lemma}\label{Lemma2.18}
With previous notations, the set

$$
\{(x, y)\in O\times O; p(x) -p(y)\in Z; \max_{i, j, k, \ell}
W_{ijk\ell} (x, y)=0\}
$$
is at most a 3-dim submanifold.

Also, for $\delta_1> \delta>0$ small enough and considering a
partition of $O$ in $\delta$-boxes $Q_\alpha$, we have
\be\label{eq2.19}
\begin{aligned}
\#\mathcal W_{\delta, \delta_1} &=\# \{(\alpha, \beta);
\big(p(Q_\alpha)-p(Q_\beta)\big)\cap Z\not= \phi \text { and }
\max_{i, j, k, \ell}
\min_{Q_\alpha\times Q_\beta}|W_{ijk\ell}(x, y)|<\delta_1\}\\
& <\delta^{-4} \delta_1^{c_1}
\end{aligned}
\ee (for some constant $c_1$ independent of $\delta_1$).
\end{lemma}

Fix $\alpha\not=\beta$ such that $p(Q_\alpha)-p(Q_\beta)\subset Z$
and $(\alpha,\beta)$ not in the exceptional set $\mathcal W=\mathcal
W_{\delta, \delta_1}$. Let $Q_\alpha =a_\alpha+U_\alpha, Q_\beta
=a_\beta+U_\beta$ where $Q_\alpha, Q_\beta\subset O$ and $U_\alpha,
U_\beta$ are $\delta$-neighborhoods of $(0, 0)$.

Appropriate coordinate changes in $x, y$ permit to bring
$\psi\big(p(a_\alpha+x)-p(a_\beta+y)\big)$ in the form
\begin{multline}\label{eq2.20}
f(x)+g(y)+x_1y_1+x_2y_2+x_1^2 Q_{11}(y)+x_1x_2 Q_{12}(y)+x_2^2Q_{22}(y) \\
 +O(|x|^2|y|^2 (|x|+|y|)\big)
\end{multline}
with \be\label{eq2.21} \max_{i, j= 1, 2} \Vert Q_{ij}\Vert>\delta_1.
\ee Next, we show
\begin{lemma}\label{Lemma2.22}
Further linear coordinate changes in $x$ and $y$ provide an
expression of the form
\begin{multline}\label{eq2.23}
f(x)+g(y)+x_1y_1 +x_2y_2+ qx_1^2 y_1^2 \\
+O\big((|x_2| \, |x|+|y_2|
\,|y|)(|x|+|y|)^2+|x|^2|y|^2(|x|+|y|)\big)
\end{multline}
with $|q|\gtrsim \delta_1$.
\end{lemma}

\begin{proof}
With $a$ a parameter to be specified, make a linear transformation
$$
x\mapsto (x_1, x_2+ax_1) \qquad y\to (y_1-ay_2, y_2)
$$
preserving the quadratic part of \eqref{eq2.20}. We obtain
$$
\begin{aligned}
f_1(x)&+g_1(y)+x_1y_1+x_2y_2+x_1^2 Q_{11}(y_1-ay_2, y_2) \\
&+x_1(x_2+ax_1)Q_{12} (y_1-ay_2, y_2)
+(x_2+ax_1)^2 Q_{22} (y_1-ay_2, y_2) \\
&+O\big(|x|^2 |y|^2 \big(|x|+|y|)\big)
\end{aligned}
$$
with bi-quadratic part \be\label{eq2.24} x_1^2[Q_{11}' (y)+a Q_{12}'
(y)+a^2 Q_{22}' (y)]+ 0(|x_2| \, |x| \, |y|^2) \ee where
$$
Q_{ij}' (y)=Q_{ij} (y_1-ay_2, y_2)
$$
satisfies, by \eqref{eq2.21} \be\label{eq2.25} \max_{i, j} \Vert
Q_{ij}' \Vert>\delta_1. \ee Clearly there is some $a=O(1)$ such that
\be\label{eq2.26} \Vert Q'\Vert=\Vert Q_{11}' +a Q_{12}'+
a^2Q_{22}'\Vert\gtrsim \delta_1. \ee Thus after this first linear
transformation, we get
\begin{multline}\label{eq2.27}
f_1(x)+g_1(y)+x_1y_1+x_2y_2+x_1^2 Q'(y_1, y_2)\\
+ O\big(|x_2| \, |x|\, |y|^2 +|x|^2 |y|^2(|x|+|y|)\big)
\end{multline}
and
$$
Q'(y_1, y_2) =q_{11} y_1^2+q_{12} y_1y_2+q_{22} y_2^2
$$
satisfying \be\label{eq2.28} \max_{i, j} |q_{ij}|\gtrsim \delta_1.
\ee

Next, make a second transformation
$$
x\mapsto (x_1-b x_2, x_2) \qquad y\mapsto(y_1, y_2+by_1)
$$
with $b=O(1)$, converting \eqref{eq2.27} to
\be
\begin{aligned}
f_2(x)&+g_2(y)+x_1y_1+x_2y_2 \\
&+  (x_1-bx_2)^2[q_{11} y_1^2+ q_{12} y_1(y_2+by_1)+ q_{22}(y_2+by_1)^2]\\
 &+O\big( |x_2| \, |x|\, |y|^2   +|x|^2 |y|^2 (|x|+|y|)\big) \\
\label{eq2.29}
=&f_2(x)+g_2(y)+x_1y_1+x_2y_2+ x_1^2 y_1^2(q_{11}+bq_{12}+b^2q_{22})\\
&+O\big(|x_2|\, |x| \, |y|^2+ |y_2| \, |y| \, |x|^2+ \, |x|^2|y|^2
(|x|+|y|)\big).
\end{aligned}
\ee By \eqref{eq2.28}, we can choose $b$ such that
\[
|q|=|q_{11}+b q_{12}+b^2 q_{22}|\gtrsim \delta_1.
\]
This proves Lemma \ref{Lemma2.22}.
\end{proof}

\section{Estimation of certain oscillatory sums}

Let $E=R^2 \in\mathbb Z_+$ be the eigenvalue.

In the preceding Section 3, we take $\delta=R^\ve$ with $\ve>0$ a
small constant and $\delta_1 =\sqrt \delta$.

Our purpose in this section is to establish nontrivial bounds on
sums of the form \be\label{eq3.1} \sum_{x\in X, y\in Y}
e^{iR[\psi(p(a_\alpha+x)-p(a_\beta+y))]} \ee where $X\subset
U_\alpha, Y\subset U_\beta$ are discrete sets of $\frac 1{\sqrt
R}$-separated points (recall that $U_\alpha, U_\beta$ are
$\delta$-neighborhoods of $(0, 0)$). The sets $X, Y$ will correspond
to diffeomorphic images of subsets of $\mathcal E = RS^2 \cap
\mathbb Z^3$ as we will explain in Section~\ref{sec:6}.

Our aim is to prove an estimate \be\label{eq3.2} |\eqref{eq3.1}| <
R^{2-\kappa} \ee for some $\kappa>0$ (independent of $R$).

The bound \eqref{eq3.2} will be derived from the following
1-dimensional inequality.

\begin{lemma}\label{Lemma3.3}
Assume $S, T\subset [0, R^{-\frac 15}]$ arbitrary discrete sets of
$\frac 1{\sqrt R}$-separated points and $0<|q| < 0(1)$. Then
\be\label{eq3.4} \Big|\sum_{s\in S, t\in T} e^{iR(st+qs^2t^2)}\Big|
< R^{\frac 35-\kappa_1}|q|^{-1} \ee for some constant $\kappa_1>0$.
\end{lemma}

Lemma~\ref{Lemma3.3} will be proven in Section~\ref{sec:5}. In this
section, we derive \eqref{eq3.2} from \eqref{eq3.4}. According to
Lemma~\ref{Lemma2.22}, we may assume for $x\in U_\alpha, y\in
U_\beta$
\be\label{eq3.5}
\begin{aligned}
\psi\big(p&(a_\alpha+x)-p(a_\beta+y)\big) = f(x)+g(y)+x_1y_1+x_2y_2+qx_1^2 y_1^2\\
&+O\big((|x_2|\, |x|+|y_2| \,|y|)(|x|+|y|)^2\big) +O\big(|x|^2|y|^2
(|x|+|y|)\big)
\end{aligned}
\ee where $|q|>\delta$.

In order to reduce the problem to a 1-dimensional setting, a further
restriction of the range of the $x, y$-variables will be performed.

Let $\bar x \in U_\alpha, \bar y \in U_\beta$ and $x= \bar x+x',
y=\bar y+y'$ with $x', y'$ suitably restricted. Write

$$
\begin{aligned}
&\psi\big(p(a_\alpha+\bar x+x') -p(a_\beta +\bar y+y')\big)=\\[6pt]
&\psi\big(p(a_\alpha+x')-p(a_\beta+y')\big) +\sum^2_{i=1} \bar x_i
A_i(\bar x, \bar y, x', y')+\sum^2_{j=1} \bar y_j B_j(\bar x, \bar
y, x', y')\overset
{\ref{eq3.5}} = \\[6pt]
&f(x')+g(y')+x'.y' \\[6pt]
&\quad +q(x_1')^2 (y_1')^2+ O\big((|x_2'|\,|x'|+|y_2'| \,
|y'|)(|x'|^2+|y'|^2|) +|x'|^2|y'|^2(|x'|+|y'|)\big)
\end{aligned}
$$
\be\label{eq3.6} +\sum^2_{i=1} \bar x_i A_i(\bar x, \bar y, x',
y')+\sum^2_{j=1} \bar y_j B_j(\bar x, \bar y, x', y'). \ee Perform
coordinate changes in $x', y'$ separately (as described in Lemma
\ref{Lemma2.1}) \be\label{eq3.7}
\begin{cases} x'= \zeta^{(1)}_{\bar x, \bar y} (x'')\\[6pt]
y'= \zeta^{(2)}_{\bar x, \bar y} (y'')
\end{cases}
\ee in order to bring \eqref {eq3.6} in the form \be \label{eq3.8}
\psi\big(p(a_\alpha+\bar x+ \zeta^{(1)}_{\bar x, \bar y} (x'')\big)-p(a_\beta+\bar y+\zeta^{(2)}_{\bar x, \bar y}(y'')\big)=\\
f_1(x'')+g_1(y'')+x''.y''+h(x'', y'') \ee where
$$
h(x'',y'')=O(|x''|^2 |y''|^2).
$$

Clearly $\zeta^{(1)}_{\bar x, \bar y}, \zeta^{(2)}_{\bar x, \bar y}$
depend real-analytically on $\bar x, \bar y$.

Also, since $|\bar x|, |\bar y|<\delta$ \be\label{eq3.9}
\begin{cases}
\zeta^{(1)}_{\bar x, \bar y} (x'')=x''+O\big((|\bar x|+|\bar y|)|x''|\big)\\[6pt]
\zeta^{(2)}_{\bar x, \bar y} (y'')=y''+O \big((|\bar x|+|\bar
y|)|y''|\big)
\end{cases}
\ee are $\delta$-perturbations of identity.

Returning to \eqref{eq3.6}, it follows that \be\label{eq3.10}
\begin{aligned}
h(x'', y'') =& q(x_1'')^2 (y''_1)^2+O\big((|x''_2| \, |x''|+|y_2''|\,|y''|) (|x''|+|y''|)^2+|x''|^2|y''|^2 (|x''|+|y''|)\big)\\[6pt]
&+O\big((|\bar x|+|\bar y|)|x''|^2|y''|^2\big)\\[6pt]
&=q''(x''_1)^2(y''_1)^2+ O\big(( |x_2''|\,|x''|+|y''_2| \, |y''|)(|x''|^2+|y''|^2)\big)\\[6pt]
&\qquad +O\big(|x''|^2|y''|^2 (|x''|+|y''|)\big)
\end{aligned}
\ee where $q'' =q+O(\delta)$, hence $|q''|>\frac 12|q|\gtrsim
\delta_1$.

Thus
$$
\begin{aligned}
\eqref{eq3.8} = &f_1(x'')+ g_1(y'')+x_1''y_1''+x_2''y_2''+q''(x_1'')^2(y_1'')^2+\\[6pt]
&O\big((|x_2''|\,|x''|+|y_2''|\, |y''|)(|x''|^2+|y''|^2)+\\[6pt]
\end{aligned}
$$
\be\label{eq3.11} \quad O\big((|x''|+|y''|)^5\big). \ee

Fix a small parameter $\tau>0$ and denote \be\label{eq3.12}
 B =[0, R^{-\frac 15-\tau}]\times  [0, R^{-\frac 12-\tau}].
\ee If we restrict $x''\in B, y''\in B$, clearly \be\label{3.13}
\begin{aligned}
\eqref{eq3.11} = &f_1(x'')+g_1(y'')+x_1''y_1''+q''(x_1'')^2 (y_1'')^2+\\[6pt]
&O(R^{-1-2\tau}+R^{-\frac 35 -\frac 12}+R^{-1-5\tau}).
\end{aligned}
\ee Hence, returning to \eqref{eq3.1}
$$
\Bigg|\sum_{\substack{x'', y''\in B\\
\zeta^{(1)}_{\bar x, \bar y}(x'') \in X-\bar x\\
\zeta^{(2)}_{\bar x, \bar y} (y'')\in Y-\bar y}}
e^{R\psi\big(p(a_\alpha+\bar x+ \zeta^{(1)}_{\bar x\bar y} (x''))
-p(a_\beta+\bar y+\zeta^{(2)}_{\bar x\bar y} (y''))\big)}\Bigg|
$$

\be\label{eq3.14}
\leq \Bigg|\sum_{\substack{x'', y''\in B\\
\zeta^{(1)}_{\bar x, \bar y}(x'') \in X-\bar x\\
\zeta^{(2)}_{\bar x, \bar y} (y'')\in Y-\bar y}} c(x'') d(y'')
c^{iR[x_1''y_1''+q''((x_1'')^2(y_1'')^2)]}\Bigg| \ee

\be\label{eq3.15} +O\big(R^{-2\tau}|X\cap[\zeta^{(1)}_{\bar x, \bar
y} (B)+\bar x]| \cdot |Y\cap [\zeta^{(2)}_{\bar x, \bar y}(B)+\bar
y]|\big) \ee with $|c(x'')|=|d(y'')|=1$.

Recall that $X, Y$ consist of $\frac 1{\sqrt R}$-separated points.
Hence also the elements of $(\zeta^{(1)}_{\bar x, \bar
y})^{-1}(X-\bar x)$ and $(\zeta^{(2)}_{\bar x, \bar y})^{-1} (Y-\bar
y)$ are $\sim \frac 1{\sqrt R}$-separated. From the definition
\eqref{eq3.12} of $B$, it follows that
$$
\begin{aligned}
S&=\pi_1 [B\cap (\zeta^{(1)}_{\bar x, \bar y})^{-1} (X-\bar x)]\\
T&= \pi_1 [B\cap (\zeta^{(2)}_{\bar x, \bar y})^{-1} (Y-\bar y)]
\end{aligned}
$$
are $\sim \frac 1{\sqrt R}$ separated.

Assuming a general estimate \eqref{eq3.4} ($\kappa_1>0$ some fixed
constant) at our disposal, we can therefore conclude that \be
\label{eq3.16} |\eqref{eq3.14}|< R^{-\kappa_1+\frac 35}\frac
1{|q''|} < R^{-\frac 12\kappa_1+\frac 35}. \ee

In conclusion, we obtain from \eqref{eq3.14}-\eqref{eq3.16}
$$
\Bigg|\sum_{\substack{x\in X \cap (\bar x+\zeta^{(1)}_{\bar x, \bar
y}(B))\\ y\in Y\cap(\bar y+\zeta^{(2)}_{\bar x, \bar y} (B))}}
e^{iR\psi(p(a_\alpha+x)- p(a_\beta+y))}\Bigg|\lesssim
$$
\be\label{eq3.17} R^{-2\tau} |X\cap \big(\bar x+\zeta^{(1)}_{\bar
x\bar y}(B)\big)|.|Y\cap \big(\bar y+\zeta^{(2)}_{\bar x\bar
y}(B)\big)| + R^{\frac 35-\frac 12\kappa_1}. \ee

Recall that $\bar x\in U_\alpha, \bar y\in U_\beta$ were arbitrarily
chosen.

Integration of \eqref{eq3.17} in $\bar x\in U_\alpha, \bar y\in
U_\beta$ gives
$$
\sum_{x\in X, y\in Y} e^{iR\psi (p(a_\alpha+x)-p(a_\beta+y))}\Bigg\{
\ \iint\limits_{U_\alpha\times U_\beta}[1_{\bar x+\zeta^{(1)}_{\bar
x\bar y}(B)} (x).1_{\bar y+\zeta^{(2)}_{\bar x \bar y}(B)}(y)] d\bar
xd\bar y\Bigg\}
$$
\be\label{eq3.18} \lesssim R^{-2\tau} \sum_{x\in X, y\in Y} \Bigg\{
\ \iint\limits_{U_\alpha\times U_\beta} [1_{\bar x+\zeta^{(1)}_{\bar
x\bar y} (B)} (x) 1_{\bar y+\zeta^{(2)}_{\bar x\bar y} (B)}
(y)]d\bar xd\bar y\Bigg\} +R^{\frac 35-\frac 12\kappa_1}. \ee Next,
we analyze the expression $\{ \ \}$.

For fixed $x, y$, consider the equations \be\label{eq3.19}
\begin{cases}x=\bar x+\zeta^{(1)}_{\bar x, \bar y} (x'')\\[6pt]
 y=\bar y+\zeta^{(2)}_{\bar x, \bar y}(y'')\end{cases}
\ee with $x'', y'' \in B$. Note that by \eqref{eq3.9}
$$
|\partial_{\bar x}\zeta^{(1)}_{\bar x\bar y}(x'')|+|\partial_{\bar
y}\partial^{(1)}_{\bar x\bar y}(x'')|+|\partial_{\bar
x}\zeta^{(2)}_{\bar x\bar y}(y'')|+ |\partial _{\bar y}
\zeta^{(2)}_{\bar x\bar y} (y'')|< O(|x''|+|y''|)<R^{-\frac 15}.
$$
Hence, by the implicit function theorem, \eqref{eq3.19} may be
rewritten as \be\label{eq3.20} (\bar x, \bar y) =\Omega_{x, y} (x'',
y'') \ee where $\Omega_{x, y}$ is a diffeomorphism from $B\times B$
to $\Omega_{x, y} (B\times B)\subset U_\alpha\times U_\beta$
(recalling again \eqref{eq3.9}).

We have
$$
\begin{aligned}
&\iint\limits_{U_\alpha\times U_\beta} [1_{\bar x+ \zeta^{(1)}_{\bar x\bar y}(B)}(x) 1_{\bar y+\zeta^{(2)}_{\bar x\bar y}(B) }(y)] d\bar x d\bar y=\\[6pt]
&\iint\limits_{\Omega_{x, y} (B\times B)} d\bar xd\bar y =\\[6pt]
\end{aligned}
$$
\be\label{eq3.21} \iint\limits_{B\times B} \Bigg|\frac{\partial
(\Omega_{xy}^{(1)}, \Omega_{xy}^{(2)})}{\partial (x'', y'')}\Bigg|
dx'' dy''. \qquad\quad \ee

It follows from \eqref{eq3.19} and the preceding that
$$
\frac{\partial\bar x}{ \partial x''}= -\frac
{\partial\zeta^{(1)}_{\bar x\bar y}}{\partial x''}+O\Big( R^{-\frac
15} \Big(\Big|\frac {\partial \bar x}{\partial x''}\Big|+\Big|
\frac{\partial\bar y}{\partial \bar x''}\Big|\Big)\Big)
$$
and hence \be\label{eq3.22} \frac{\partial\bar x}{\partial x''} =
-1+O(\delta)+O(R^{-\frac 15}). \ee

Similarly \be\label{eq3.23} \frac{\partial\bar x}{\partial y''}
=O(R^{-\frac 15}) \ee \be\label{eq3.24} \frac{\partial\bar
y}{\partial x''} =O(R^{-\frac 15}) \ee \be\label{eq3.25}
\frac{\partial\bar y}{\partial y''}= -1+O(\delta)+O(R^{-\frac 15}).
\ee From \eqref{eq3.22}-\eqref{eq3.25}
$$
D\Omega_{x, y} =-1+O(\delta)
$$
implying
$$
\frac{\partial(\Omega^{(i)}_{xy}, \Omega^{(2)}_{xy})}{\partial (x'',
  y'')}
=1+O(\delta)
$$
and \be\label{eq3.26} \eqref{eq3.21} =\omega(x, y)|B|^2 \ee where
\be\label{eq3.27} \omega(x, y)=1+O(\delta) \ee is a smooth function
of $x, y$.

Substituting \eqref{eq3.21}, \eqref{eq3.26} in \eqref{eq3.18} gives
\be\label{eq3.28} \Bigg|\sum_{x\in X, y\in Y}
e^{iR\psi(p(a_\alpha+x)-p(a_\beta+y))} \ \omega(x, y)\Bigg| \lesssim
R^{-2\tau} |X|\cdot |Y| +R ^{2-\frac 12 \kappa_1+4\tau} \ee
recalling \eqref{eq3.12}.

It remains to remove the function $\omega(x, y)$ in \eqref{eq3.28}.

First observe that \eqref{eq3.28} formally implies that
\be\label{eq3.29} \Bigg|\sum_{x\in X, y\in Y}
e^{iR\psi(p(a_\alpha+x)-p(a_\beta+y))}\omega(x, y)u(x)v(y)\Bigg| <
R^{-2\tau} |X|\cdot |Y| +R^{2-\frac 12 \kappa_1+4\tau} \ee whenever
$u, v$ are functions on $\mathbb R^2$ satisfying $|u|\leq 1, |v|\leq
1$.

Since $\omega$ is a smooth function of $(x, y)$ satisfying
\eqref{eq3.27}, it follows that $\frac 1\omega\in
L^\infty\hat\otimes L^\infty$, thus $\frac
1\omega=\Sigma\lambda_\ell(u_\ell\otimes v_\ell)$ where $\Vert
u_\ell\Vert _{\infty}, \Vert v_\ell\Vert_{\infty}\lesssim 1$ and
$\Sigma|\lambda_\ell|<C$. Hence, by convexity, \eqref{eq3.29}
implies \be\label{eq3.30} \Bigg|\sum_{x\in X, y\in Y}
e^{iR\psi(p(a_\alpha+x) - p(a_\beta
  +y))}\Bigg|\lesssim R^{-2\tau} |X|\cdot |Y|+R^{2-\frac 13\kappa_1}
\ee taking $\tau>0$ small enough.

This gives an inequality of the type \eqref{eq3.2}. The sets
$X\subset U_\alpha, Y\subset U_\beta$ are arbitrary sets of $\frac
1{\sqrt R}$-separated points.

Returning to \eqref{eq2.19}, we proved that if $X, Y$ are $\frac
1{\sqrt R}$-separated points in $O$, then \be\label{eq3.31}
\Bigg|\sum_{\substack{x\in X U Q_\alpha\\ y\in Y\cap Q_\beta}}
e^{iR\psi(p(x) -p(y))}\Bigg|< R^{-2\tau}|X\cap Q_\alpha| \,|Y\cap
Q_\beta|+ R^{2-\frac 13\kappa_1} \ee provided $(\alpha,
\beta)\not\in \mathcal W=\mathcal W_{\delta, \delta_1}$ and
$p(Q_\alpha)-p(Q_\beta)\subset Z$. Here $\tau, \kappa_1>0$ are
constants and $\delta=\delta^2_1=R^{-\ve}, \ve>0$ sufficiently
small.

Summation of \eqref{eq3.31} over $\alpha, \beta$ gives
\be\label{eq3.32} \sum_{\substack{p(Q_\alpha)-p(Q_\beta)\subset Z\\
(\alpha, \beta)\not\in \mathcal W}} \Bigg|\sum_{\substack{x\in X\cap
Q_\alpha\\ y\in Y\cap Q_\beta}} e^{iR\psi (p(x) -
p(y))}\Bigg|\lesssim R^{2-2\tau}+ \delta^{-4} R^{2-\frac
13\kappa_1}< R^{2-\kappa_2}. \ee Recalling \eqref{eq2.19}, it
follows that \be\label{eq3.33}
\begin{split}
\Bigg|\sum_{\substack{x\in X, y\in Y\\ p(x) -p(y)\in Z}}
e^{iR\psi(p(x) - p(y))}\Bigg| & < R^{2-\kappa_2}+\delta^{-4+\frac 12
c_1}
\max|X \cap B_\delta| \max|Y\cap B_\delta| \\
&< R^{2-\kappa_2}+\delta^{\frac 12 c_1} R^2 \\
&< R^{2-\kappa_3}
\end{split}
\ee since the points in $X, Y$ are $\frac 1{\sqrt R}$-separated.

Equivalently, considering sets $\mathcal X, \mathcal Y\subset RS^2$
consisting of $\sqrt R$-separated points and such that $\mathcal
X\cup\mathcal Y$ is contained in a cap of size $cR$ ($c>0$ a
constant depending on $\Sigma$) we have \be\label{eq3.34}
\Bigg|\sum_{x\in\mathcal X, y\in\mathcal Y, x-y\in Z}
e^{i\psi(x-y)}\Bigg|< R^{2-\kappa_3}. \ee

Thus (conditional to Lemma \ref{Lemma3.3}) we proved the following

\begin{lemma}\label{Lemma3.35}
Let $\mathcal X, \mathcal Y\subset RS^2$ consist of $\sqrt
R$-separated points. Then \be\label{eq3.37} \Bigg|\sum_{\substack
{x\in\mathcal X, y\in\mathcal Y\\ x-y \in Z, |x-y|< cR}}
e^{i\psi(x-y)}\Bigg| < R^{2-\gamma} \ee for some $c>0, \gamma>0$
depending on $\psi$ (hence on $\Sigma$).
\end{lemma}

\section{An Exponential Sum Estimate}\label{sec:5}

We prove the key inequality Lemma \ref{Lemma3.3}.

Let $R$ be large enough, $0<|q|<O(1)$ and $S, T\subset [0, R^{-\frac
15}]$ arbitrary discrete sets of $\frac 1{\sqrt R}$-separated
points. Denoting \be\label{eq4.1} \mathfrak S =\sum_{s\in S, t\in T}
e^{iR(st+qs^2t^2)} \ee application of the Cauchy-Schwartz inequality
twice gives
\be
\begin{aligned}\label{eq4.2}
|\mathfrak S|^4&\leq |S|^2 |T|^2\Bigg| \sum_{\substack{s, s_1\in S\\
t, t_1\in T}}
e^{iR((s-s_1)(t-t_1)+q(s^2-s^2_1)(t^2-t_1^2)}\Bigg| \\ 
&=|S|^2 |T|^2\Bigg|\sum_{z, w} e^{i[R^{3/5} z_1w_1+q R^{1/5}
    z_2w_2]}\mu(z)\nu(w)\Bigg|
\end{aligned}
\ee where $z=(z_1, z_2), w=(w_1, w_2)$ and $\mu, \nu$ are discrete
measures on $[-1, 1]^2$ defined by \be\label{eq4.3} \mu(z) =\#
\Bigg\{(s, s_1) \in S\times S \ \Big| \ \begin{matrix} s-s_1
=R^{-\frac 15}z_1\\ s^2-s^2_1 =R^{-\frac 25} z_2\end{matrix}\Bigg\}
\ee and similarly for $\nu$. Thus \be\label{eq4.4} \sum \mu(z),
\sum\nu(w) \leq R^{3/5}. \ee Fix $0<\theta<\frac 1{10}$. Since $S$
is $\frac 1{\sqrt R}$-separated \be\label{eq4.5}
\sum_{|z_1|<R^{-\theta}} \mu(z) <|S| R^{\frac 3{10} -\theta}
\lesssim R^{3/5-\theta} \ee and similarly for $\nu$. Hence in
\eqref{eq4.2},
\begin{multline}\label{eq4.6}
\Bigg|\sum_{z, w} e^{i[R^{3/5} z_1w_1+q R^{1/5}
    z_2w_2]}\mu(z)\nu(w)\Bigg| \\
<  \left| \sum_{\substack {z, w\\ |z_1|, |w_1|>R^{-\theta}}}
e^{i[R^{3/5}z_1
    w_1+qR^{1/5} z_2w_2]} \mu(z)\nu(w) \right|+ O(R^{\frac 65-\theta})
\end{multline}
In order to bound the RHS of \eqref{eq4.6}, we apply the following
general bilinear estimate.

\begin{lemma}\label{Lemma4.7}
$$
\begin{aligned}
\Big|\sum_{z, w} e^{i(R_1z_1 w_1+R_2 z_2w_2)} \mu(z)
\nu(w)\Big|&\lesssim
(R_1R_2)^{\frac 12}\Big(\sum\mu(z)\Big)^{\frac12} \Big(\sum\nu(w)\Big)^{\frac 12}.\\
&\Big[\max_{\xi_1, \xi_2} \mu\Big(B\Big(\xi_1, \frac 1{R_1}\Big)\times B\Big(\xi_2, \frac 1{R_2}\Big)\Big)\Big]^{\frac 12}.\\
& \Big[\max_{\xi_1, \xi_2} \nu \Big(B\Big(\xi_1, \frac
1{R_1}\Big)\times B\Big(\xi_2, \frac 1{R_2}\Big)\Big)\Big]^{\frac
12}.
\end{aligned}
$$
\end{lemma}


\begin{proof}
Denoting $P_\ve$ an approximate identity on $\mathbb R$, the left
side equals \be\label{eq4.8}
\begin{aligned}
\Big|\sum_w\hat\mu & (R_1w_1, R_2 w_2)\nu(w)\Big|\lesssim \sum_w
\Big|\Big(\mu* (P_{\frac 1{R_1}}\otimes P_{\frac
1{R_2}})\Big)^\wedge
(R_1w_1, R_2w_2)\Big|\nu(w)\\[6pt]
&\lesssim R_1R_2 \iint \Big|\big(\mu*(P_{\frac 1{R_1}} \otimes
P_{\frac 1{R_2}})\big)^\wedge (R_1\xi_1,
R_2\xi_2)\Big|\nu\Big(B\Big(\xi_1, \frac 1{R_1}\Big)
\times B\Big(\xi_2, \frac 1{R_2}\Big)\Big) d\xi_1 d\xi_2\\[6pt]
&\leq (R_1R_2)^{\frac 12}\Vert\mu*(P_{\frac 1{R_1}}\otimes P_{\frac
1{R_2}})\Vert_2.\Big\{\iint\Big[\nu\Big(B\Big(\xi_1,
\frac 1{R_1}\Big) \times B\Big(\xi_2, \frac 1{R_2}\Big)\Big) \Big]^2 d\xi_1 d\xi_2\Bigg\}^{1/2}\\[6pt]
&\lesssim (R_1R_2)^{-\frac 12} \Vert\mu*(P_{\frac 1{R_1}}\otimes
P_{\frac 1{R_2}})\Vert_2 \Vert\nu*(P_{\frac 1{R_1}}\otimes P_{\frac
1{R_2}})\Vert_2
\end{aligned}
\ee where $\Vert \ \Vert_2$ refers to $L^2([-1, 1]^2)$.

Next \be\label{eq4.9}
\begin{aligned}
\Vert \mu*(P_{\frac 1{R_1}}\times P_{\frac 1{R_2}})\Vert_2
&\leq\Vert\mu* (P_{\frac 1{R_1}} \times P_{\frac
1{R_2}})\Vert_1^{\frac 12} \Vert\mu* (P_{\frac 1{R_1}}\times
P_{\frac 1{R_2}})\Vert^{\frac
12}_\infty\\
&\sim\Big[\sum\mu(z)\Big]^{\frac 12} (R_1R_2)^{\frac 12}
\Big[\max_\xi\mu \Big(B\Big(\xi_1, \frac 1{R_1}\Big)\times
B\Big(\xi_2, \frac 1{R_2}\Big)\Big)\Big]^{\frac 12}
\end{aligned}
\ee and similarly for $\nu$.

Substitution of \eqref{eq4.9} in \eqref{eq4.8} proves the Lemma.
\end{proof}

Now apply Lemma \ref{Lemma4.7} to evaluate \eqref{eq4.6}. Thus
$R_1=R^{3/5}$, $R_2=qR^{1/5}$.

It remains to bound for $\xi= (\xi_1, \xi_2)\in[-1, 1]^2$, $|\xi_1|>
R^{-\theta}$, the quantity \be\label{eq4.10}
\begin{aligned}
&\mu\Big(B\Big(\xi_1, \frac 1{R_1}\Big)\times B\Big(\xi_2, \frac 1{R_2}\Big)\Big)=\\[6pt]
&\#\Big\{ (s, s_1)\in S\times S; |R^{-\frac 15}\xi_1 -(s-s_1)|<
R^{-4/5} \text { and } |R^{-\frac 25}\xi_2-(s^2-s^2_1)|<\frac
1qR^{-3/5}\Big\}.
\end{aligned}
\ee

From the equations, one gets
$$
\Big|R^{-\frac 15} \frac{\xi_2}{\xi_1}-(s+s_1)\Big|< \frac 1q
R^{-\frac 25+\theta} + R^{-\frac 35+\theta} < \frac 2q R^{-\frac
25+\theta}
$$
and \be\label{eq4.11} \Big|R^{-\frac 15}
\Big(\xi_1+\frac{\xi_2}{\xi_1}\Big)-2s\Big| <\frac 3q R^{-\frac
25+\theta}. \ee Since the elements of $S$ are $\frac 1{\sqrt
R}$-separated, \eqref{eq4.11} restricts $s$ to at most $\frac cq
R^{\frac 1{10}+\theta}$ values. Hence \be\label{eq4.12}
|\eqref{eq4.10}|\lesssim \frac 1q R^{\frac 1{10}+\theta}. \ee
 From Lemma \ref{Lemma4.7}, recalling \eqref{eq4.4} and
 \eqref{eq4.12}, we obtain
\be\label{eq4.13} |\eqref{eq4.6}|\lesssim (qR^{4/5})^{\frac 12}
R^{3/5} \frac 1q R^{\frac 1{10} +\theta}\lesssim \frac 1{\sqrt q}
R^{\frac{11}{10}+\theta}. \ee Hence
$$
|\eqref{eq4.2}|\lesssim \frac 1{\sqrt q} R^{\frac {11}{10}+\theta}
+R^{\frac 65 -\theta}
$$
and
$$
|\mathfrak S|^4 \lesssim R^{6/5} \Big(\frac 1{\sqrt q}
R^{\frac{11}{10}+\theta} + R^{\frac 65-\theta}\Big).
$$
An appropriate choice of $\theta$ gives \be\label{eq4.14} |\mathfrak
S| < R^{\frac{47}{80}} q^{-\frac 14}. \ee This proves Lemma
\ref{Lemma3.3} with $\kappa_{1} =\frac 1{80}$. \qed

\bigskip
\section{Mean Equidistribution Property of Lattice Points}
\label{sec:6}

Let $\mathcal E =\mathbb Z^3 \cap RS^2, R=\sqrt E$. Recall that
\be\label{eq5.1} |\mathcal E|\ll R^{1+\ve} \text { for all $\ve>0$.}
\ee

In order to apply Lemma \ref{Lemma3.35} with $\mathcal X, \mathcal
Y\subset\mathcal E$, we recall Lemma~\ref{IntLemma5.6}, which states
that
\begin{lemma}\label{Lemma5.6}
Let $\{C_\alpha\}$ be a partition of $RS^2$ in cells of size $\sqrt
R$. Then \be\label{eq5.7} \sum_\alpha|C_\alpha\cap\mathcal E|^2 \ll
R^{1+\ve} \text { for all $\ve >0$.} \ee
\end{lemma}

Thus \eqref{eq5.7} express the desired separation property in some
averaged sense. To obtain sets that are $\sqrt R$-separated, proceed
as follows. Fix $\ve'>0$ and let
$$
\mathcal E'=\bigcup_{|C_\alpha\cap \mathcal E|> R^{\ve'}} (\mathcal
E\cap C_\alpha).
$$
It follows from \eqref{eq5.7} that \be\label{eq5.8} |\mathcal E'|<
R^{-\ve'} \sum|\mathcal E\cap C_\alpha|^2 \ll R^{1+\ve-\ve'}<
R^{1-\frac{\ve'}2}. \ee Also \be\label{eq5.9} \mathcal
E\backslash\mathcal E' = \bigcup_{s< R^{\ve'}}\mathcal X_s \ee with
each set $\mathcal X$ consisting of $\sqrt R$-separated points.

From \eqref{eq3.37} \be\label{5.9}
\Bigg|\sum_{\substack{x\in\mathcal X_s, y\in\mathcal X_t\\ x-y\in Z,
|x-y|\ll cR}} e^{i\psi(x-y)}\Bigg|< R^{2-\gamma} \ee for some
$\gamma>0$.

Therefore, from \eqref{eq5.1}, \eqref{eq5.8} \be\label{eq5.10}
\Bigg|\sum_{\substack{x\in\mathcal E_1, y\in\mathcal E_2\\ x-y\in Z,
|x-y|<cR}} e^{i\psi(x-y)} \Bigg| \ll R^{2-\gamma+2\ve'}+ 2| \mathcal
E'|\, |\mathcal E| < R^{2-\gamma+2\ve'}+ R^{2+\ve-\frac{\ve'}2} \ee
if $\mathcal E_1, \mathcal E_2 \subset \mathcal E$.
\medskip

Hence
\begin{lemma}\label{Lemma5.11}
There is a constant $\gamma_1>0$ (independent of $R$) such that
\be\label{eq5.11} \Bigg| \sum_{\substack{x\in \mathcal E_1, y\in
\mathcal E_2\\ x-y\in Z, |x-y|< cR}} e^{i\psi(x-y)}\Bigg|<
R^{2-\gamma_1} \ee whenever $\mathcal E_1, \mathcal E_2
\subset\mathcal E=(RS^2 \cap \mathbb Z^3)$.
\end{lemma}

This is our main estimate to handle `large distances' $|x-y|>
R^{1-\ve}$.

\bigskip
\section
{Restriction Upper Bound}\label{sec:restriction upper bd}

\begin{theorem}\label{Theorem1}
Let $\Sigma\subset \mathbb T^3$ be a real-analytic 2-dim submanifold
with non-vanishing curvature and let $\sigma$ be its surface
measure. There is a constant $K_\Sigma>0$ such that
\be\label{eq6.14} \int_\Sigma |\vp|^2d\sigma \leq K_\Sigma
\Vert\vp\Vert_2^2 \ee for all eigenfunctions $\vp$ on $\mathbb T^3$.
\end{theorem}

Let \be\label{eq6.1} \vp = \sum_{n\in\mathcal E} a_n e^{ix.n} \
\text { with } \ \sum|a_n|^2=1 \ee and
$$
\mathcal E =\{n\in\mathbb Z^3; |n|^2= E=R ^2\}.
$$
Then \be \label{eq6.3}
\begin{aligned}
\int_\Sigma |\vp|^2 d\sigma &=\sum_{m, n\in\mathcal E} a_m \bar a_n\int_\Sigma e^{i(m-n).x} \sigma(dx)\\
&= ||\vp||_2^2 \mbox{area}(\Sigma) + \sum_{k\geq 0}
\sum_{\substack{m, n\in\mathcal E\\ 2^k\leq |m-n|< 2^{k+1}}} a_m\bar
a_n \int_\Sigma e^{i(m-n).x} \sigma(dx)\;.
\end{aligned}
\ee Considering local coordinate charts, we can assume $\Sigma$ is
parametrized as in \eqref{eq1.1}.
 From \eqref{eq1.10}, if $m\neq n$ then
\be \label{eq6.4} \int_\Sigma e^{i(m-n).x} d\sigma =\frac 1{|m-n|}
\eta\Big(\frac{m-n}{|m-n|}\Big) e^{i\psi(m-n)} + O\Big(\frac
1{|m-n|^2}\Big) \ee with $\eta$ a smooth function on $S^2$.

First we bound  the contribution of the error term in \eqref{eq6.4}
by writing \be \label{eq6.5}
\begin{aligned}
\sum_{m, n\in \mathcal E, 2^k\leq |m-n|<2^{k+1}}& |a_m| \, |a_n| \frac 1{|m-n|^2} \lesssim 4^{-k} \sum_\alpha \Big(\sum_{m\in\mathcal E\cap C_\alpha}|a_m|\Big)^2\\
&\leq 4^{-k}\big(\max_\alpha|\mathcal E\cap C_\alpha|\big) \ \sum_{\alpha}\Big(\sum_{m\in\mathcal E\cap C_\alpha} |a_m|^2\Big)\\
&\lesssim 4^{-k}\max_\alpha |\mathcal E\cap C_\alpha|
\end{aligned}
\ee where $\{C_\alpha\}$ is a partition of $RS^2$ in cells of size
$\sim 2^k$. Thus we need to bound $|\mathcal E\cap C_r|$, where
$C_r\subset RS^2$ is a cap of size $r$.

If $r< cR^{1/4}$, a Jarnik type theorem implies that $\mathcal E\cap
C_r$ lies in a $2$-dim affine plane $H$. Projection of $H\cap RS^2$
on one of the coordinate planes $xy, yz, zx$ gives a non-degenerate
ellipse of size $\sim r$. Another application of the classical
Jarnik theorem in the plane shows that certainly \be\label{eq6.6}
|\mathcal E\cap C_r|<Cr^{2/3}. \ee For $r$ arbitrary, one has the
(easy) linear bound (see Lemma~\ref{Lemma8.10}) \be\label{eq6.7}
|\mathcal E\cap C_r|\ll R^\ve (1+r). \ee
 From \eqref{eq6.6}, \eqref{eq6.7} we get
\be\label{eq6.9} |\mathcal E\cap C_r|\ll r^{1+\ve}. \ee Substituting
\eqref{eq6.9} in \eqref{eq6.5} gives $2^{-k(1-\ve)}$, which is
summable in $k$.

\bigskip

Consider next the contribution of the main term in \eqref{eq6.4}. We
make two separate estimates. The first treats the case
$2^k<R^{1-\ve_0}$ ($\ve_0>0$ some small constant) and the second
$2^k> R^{1-\ve_0}$. The following improvement of the lattice point
estimates \eqref{eq6.6}, \eqref{eq6.7}, which will be proven in
Section~\ref{sec:lattice pts} (Lemma~\ref{Lemma8.13}),  is crucial
to our analysis: For $0<\eta<1/16$ \be\label{Lemma6.8} |\mathcal
E\cap C_r|\ll R^\ve \Big(1+r\Big(\frac rR\big)^\eta\Big). \ee

\medskip

\medskip
\noindent {\bf The case $2^k <R^{1-\ve_0}$:} Ignoring again the
phase function, and arguing as in \eqref{eq6.5} gives
$$
\sum_{m, n\in \mathcal E, 2^k\leq |m-n|<2^{k+1}} |a_m| \, |a_n| \frac 1{|m-n|} \lesssim 2^{-k}\max_\alpha|\mathcal E\cap C_\alpha|\\
$$
\be \label{eq6.10}
<\begin{cases} C2^{-\frac 13k} \text { if } 2^k<cR^{1/4}\\[6pt]
R^\ve\Big[2^{-k}+\Big(\frac{2^k}R\Big)^\eta\Big] \text { if }
cR^{1/4} <2^k<R^{1-\ve_0}\end{cases} \ee invoking \eqref{eq6.6},
\eqref{Lemma6.8}. These bounds are again conclusive.

\noindent {\bf The case $R^{1-\ve_0}<2^k<R$:} This requires a more
subtle argument involving the oscillatory factor $e^{i\psi(n-n)}$ in
\eqref{eq6.4}.

Let $b$ be a smooth (radial) function on $\mathbb R^3$ satisfying
$$
\begin{cases}
b(x)&=0 \text { if } |x|<\frac 12\\
b(x)&=1 \text { if } |x|>1\end{cases}
$$
and estimate \be \label{eq6.11} \sum_{m, n\in \mathcal E} a_m\bar
a_n \, b\Big(\frac{m-n}{R^{1-\ve_0}}\Big) \frac 1{|m-n|}
\eta\Big(\frac {m-n}{|m-n|}\Big) e^{i\psi(m-n)}. \ee

Decompose $\mathcal E=\mathcal E_\ell \coprod \mathcal E_s$, where
$$
\mathcal E_\ell =\{ m\in\mathcal E; |a_m| > R^{-\frac 12+2\ve_0}\}.
$$
Since $\sum_n |a_n|^2=1$, we have $|\vE_\ell|\leq R^{1-4\ve_0}$.

Then we estimate \be\label{eq6.12} |\eqref{eq6.11}|\lesssim
\sum_{(m, n)\in(\mathcal E\times\mathcal
  E)\backslash (\mathcal E_s\times\mathcal E_s)} |a_m| \, |a_n|\frac
1{R^{1-\ve_0}} \ee \be\label{eq6.13} +\Big|\sum_{m, n\in\mathcal
E_s} a_m\bar a_n \ b\Big(\frac{m-n}{R^{1-\ve_0}}\Big) \frac 1{|m-n|}
\eta \Big(\frac {m-n}{|m-n|}\Big) e^{i\psi(m-n)}\Big|. \ee By
Cauchy-Schwarz, \eqref{eq6.12} is bounded by
$$
\frac 1{R^{1-\ve_0}} |\mathcal E|^{\frac 12} \, |\mathcal
E_\ell|^{\frac 12} \ll \frac 1{R^{1-\ve_0}} R^{\frac 12+o(1)}
R^{\frac 12 -2\ve_0}< R^{-\frac{\ve_0}{2}}.
$$
The term \eqref{eq6.13} is bounded using Lemma \ref{Lemma5.11} and a
partition of unity. This gives an estimate of the form
$$
R^{C\ve_0} \frac 1{R^{1-\ve_0}} \frac 1{R^{1-4\ve_0}} R^{2-\gamma_1}
< R^{-\frac 12\gamma_1}
$$
if $\ve_0$ is chosen sufficiently small.

Hence, we have proven Theorem~\ref{Theorem1}. \qed

\bigskip
\section
{Restriction Lower Bounds}

We prove
\begin{theorem}\label{Theorem2x}
Given $\Sigma$ as in Theorem~\ref{Theorem1}, there is $E_0\in\mathbb
Z_+$ and a constant $k_\Sigma>0$ such that \be\label{eq7.1}
\int_\Sigma |\vp|^2 d\sigma \geq  k_\Sigma ||\vp||_2^2 \ee whenever
$\vp$ is an eigenfunction with eigenvalue $E>E_0$.
\end{theorem}

Let $\vp =\sum_{n\in\mathcal E} a_n e^{in.x}$ with $\mathcal E =
RS^2 \cap\mathbb Z^3, R=\sqrt E$ and $\Sigma|a_n|^2=1$.

Write
\begin{eqnarray}
\label{eq7.2} \int_\Sigma |\vp|^2d\sigma &=& \sum_{\substack {m,
n\in\mathcal E\\ |m-n|<R^{\frac 15}}} \ a_m\bar a_n\int_\Sigma
e^{i(m-n)x}
\sigma(dx) \\
&+& \sum_{|m-n|> R^{\frac 15}}\cdots
\label{eq7.3}\\
&=& \eqref{eq7.2}+\eqref{eq7.3} \;. \nonumber
\end{eqnarray}

From the upper-bound analysis in Section~\ref{sec:restriction upper
bd}, we have \be\label{eq7.4} |\eqref{eq7.3}|< R^{-\delta} \ee for
some $\delta>0$.

Next, we analyze \eqref {eq7.2}.

Introduce a graph $\mathcal G$ on $\mathcal E$ defined by
$$
\mathcal G =\{(m, n) \in\mathcal E, |m-n|< R^{1/5}\}.
$$
Let $\{\mathcal E_\alpha\}$ be the connected components of $\mathcal
G$.
\medskip

\begin{lemma}\label{Lemma7.5} For each $\alpha$, the set $\mathcal E_\alpha$ is contained in an affine plane.
\end{lemma}

\begin{proof}
We may obviously assume $\#\mathcal E_\alpha\geq 3$ and hence there
is a subset $\mathcal F_0\subset\mathcal E_\alpha, \# \mathcal
F_0=3$ and diam$\,\mathcal F_0 <2R^{1/5}$.

Let $H=\langle\mathcal F_0\rangle$ be the affine plane spanned by
$\mathcal F_0$.

Write
$$
\mathcal E_\alpha=\bigcup_{j} \mathcal F_j
$$
where
$$
\mathcal F_{j+1} =\{m\in\mathcal E; \dist (m, \mathcal F_j)<
R^{1/5}\}.
$$

We show inductively that $\mathcal  F_j\subset H$ for each $j$.

For $j<R^{\frac 1{100}}$, dist\,$(\mathcal F_j, \mathcal F_0)<
jR^{1/5} \ll R^{1/4}$ and Jarnik's theorem implies that $\mathcal
F_j$ is coplanar. Hence $\mathcal F_j \subset H$. Next, assume
$j_0\geq R^{\frac 1{100}}, \mathcal F_{j_0} \subset H$ and $\mathcal
F_{j_0+1}\not= \mathcal F_{j_0}$. If $x_{j_0+1} \in \mathcal
F_{j_0+1}$, there are clearly $x\in\mathcal F_{j_0}$ and $y,
z\in\mathcal F_{j_0}$ satisfying
$$
|x_{j_0+1} -x|< R^{\frac 15}
$$
and
$$
x, y, z \text{ are distinct and diam\,$\{x, y, z\}\lesssim
R^{1/5}$}.
$$
(we use here that $\#\mathcal F_0=3$).

Since diam\,$\{x, y, z, x_{j_0+1}\}\lesssim R^{\frac 15}$, it
follows again from Jarnik that $x, y,z, x_{j_0+1}$ are coplanar and
hence
$$
x_{j_0+1} \in \langle x, y,z\rangle =H.
$$
This proves Lemma \ref{Lemma7.5}.
\end{proof}

Returning to \eqref{eq7.2}, it follows from definition of $\mathcal
G$ that
$$
\begin{aligned}
\eqref{eq7.2} &=\sum_\alpha\sum_{\substack{m, n \in \mathcal
E_\alpha\\ |m-n|<R^{1/5}}} a_m \bar a_n \int_\Sigma
e^{i(m-n)x}d\sigma
\\[6pt]
&=\sum_\alpha\int_\Sigma \Big|\sum_{m \in\mathcal E_\alpha} a_m \,
e^{imx}\Big|^2 d\sigma
\end{aligned}
$$
\be\label{eq7.6} +0\Big(\sum_\alpha \ \sum_{\substack{m, n\in
\mathcal E_\alpha\\ |m-n|\geq R^{1/5}}} \ \frac {|a_m| \,
|a_n|}{|m-n|}\Big). \ee The last term may be bounded by $R^{-\frac
1{31}}$, as seen as follows. Estimate by
$$
\eqref{eq7.6}< \sum_{2^k> R^{1/5}} 2^{-k} \Big(\max_C |\{(m, n)\in
C\times C; m, n\in \mathcal E_\alpha \text { for some
$\alpha$}\}|\Big)^{\frac 12}
$$
where the max is taken over all $2^k$-caps $C$. For $C$ an $r$-cap,
\eqref{eq6.6}, \eqref{eq6.9} imply that
$$
|\{\cdots\}|\leq |C\cap \mathcal E|. \max_\alpha|C\cap \mathcal
E_\alpha|\ll r^{1+2/3+\ve}
$$
hence the claim.

To prove Theorem~\ref{Theorem2x}, it  will therefore suffice to show
the following:

\begin{lemma}\label{Lemma7.7}
Let $\vp=\sum_{m\in\mathcal F} a_m e^{im.x}$, $\sum_{m\in\mathcal
F}|a_m|^2=1$, where $\mathcal F\subset\mathcal E$ consists of
coplanar points. Then \be\label{eq7.8} \int_\Sigma |\vp|^2 d\sigma>k
\ee where $k>0$ is independent of $E$.
\end{lemma}

\begin{proof}
Let $H=\langle\mathcal F\rangle$ be the plane containing $\mathcal
F$ and $\pi$ the orthogonal projection on $H_0$=plane parallel to
$H$ through 0.
 Clearly, fixing any element $m_0\in\mathcal F$, we have
$$
\int_\Sigma \Big|\sum_{m\in\mathcal F} a_m e^{imx}\Big|^2d\sigma=\int_\Sigma \Big|\sum_{m\in\mathcal F} a_m e^{i(m-m_0).x}\Big|^2 d\sigma\\
$$\be
\label{eq7.9}
 = \int_\Sigma \Big|\sum_{m\in\mathcal F} a_m e^{i(m-m_0).\pi(x)}\Big|^2 d\sigma.
\ee Let $\pi[\sigma]$ be the image measure of $\sigma$ under the map
$\pi\big|_\Sigma:\Sigma\to H_0$. Since $\Sigma$ has non-vanishing
curvature, there is a disc $B_\rho\subset H_0$ ($\rho$-independent
of $H_0$) such that \be\label{eq 7.10} \pi[\sigma] \geq \mu_{H_0}
|_{B_\rho} \ee where $\mu_{H_0}$ is a Lebesque measure on $H_0$.
Hence \be\label{eq7.11} \eqref{eq7.9}\geq
\int_{B_\rho}\Big|\sum_{m\in\mathcal F} a_m e^{i(m-m_0).y}\Big|^2
dy. \ee Since $\mathcal F\subset RS^2 \cap H, \mathcal F_0=\mathcal
F-m_0$ lies on a translate of some circle \hfill\break $\{x\in H_0;
|x|= r\}, r\leq R$.

Let $r_0$ be sufficiently large (to be specified later).

We distinguish two cases.
\medskip

\noindent {\bf Case 1:} $r<r_0$

Since $\Sigma_{m\in\mathcal F} a_m e^{i(m-m_0)\cdot y}$ is a nonzero
trigonometric polynomial with frequencies  $|m-m_0|< r_0$, it
follows that \be\label{eq7.12} \eqref{eq7.11}> C(\rho, r_0). \ee

\medskip
\noindent {\bf Case 2.} $ r\geq r_0. $

By Jarnik's theorem,
$$
\mathcal F_0=\bigcup_\alpha \mathcal F_\alpha
$$
where $\# \mathcal F_\alpha\leq 2$ and dist\,$(\mathcal F_\alpha,
\mathcal F_\beta)\gtrsim r^{1/3}$ for $\alpha\not=\beta$. Let $\eta$
be a smooth bumpfunction, supp $\eta\subset B_\rho$. Then
$$
\eqref{eq7.11} \geq \int\Big|\sum_{m\in\mathcal F} a_m e^{i(m-m_0).y} \Big|^2 \eta(y) dy\\
$$
\be\label{eq7.13} =\sum_\alpha\int\Big|\sum_{m\in\mathcal F_\alpha}
a_m e^{i(m-m_0).y}\Big|^2 \eta(y) dy \ee \be\label{eq7.14}
+\sum_{\substack{\alpha\not=\beta\\ m\in\mathcal F_{\alpha},
n\in\mathcal F_\beta}} a_m\bar a_n\int e^{i(m-n).y} \eta(y)dy \ee
and \be\label{eq7.15} |\eqref{eq7.14}|\leq \sum_{\substack{m, n\in
\mathcal F\\ |m-n| \gtrsim r^{1/3}}} |a_m| \, |a_n|
\frac{C(\rho)}{|m-n|^{10}} < \frac{C(\rho)}{r_0}. \ee

Since $\# \mathcal F_\alpha \leq 2$, arguing as in the proof of the
case $d=2$ (see the Introduction), we have  for each $\alpha$
\be\label{eq7.16} \int\Big|\sum_{m\in \mathcal F_\alpha} a_m
e^{i(m-m_0).y}\Big|^2 \eta(y) dy> c(\rho)\sum_{m\in\mathcal
F_\alpha}|a_m|^2 \ee
and thus \be\label{eq7.17} \eqref{eq7.13} >c(\rho)
\Big(\sum_{m\in\mathcal F}|a_m|^2\Big) =c(\rho). \ee From
\eqref{eq7.13}--\eqref{eq7.17}
$$
\eqref{eq7.11}> c(\rho)-\frac{C(\rho)}{r_0}>\frac 12 c(\rho)
$$
for appropriate $r_0$.
\end{proof}
This concludes the proof of Theorem~\ref{Theorem2x}.

\section{Intersection of Nodal Sets with Submanifolds}\label{sec:intersection}

We start by recalling the following result from \cite{B-R1}.

\begin{theorem}\label{TheoremB-R1}\cite{B-R1}.

Let $\Sigma\in\mathbb T^d$ be a real analytic, codimension one,
hypersurface with nowhere-vanishing Gauss curvature. Then there is
some $E_\Sigma>0$ so that if $E>E_\Sigma$, then $\Sigma$ cannot be
part of the nodal set of any eigenfunction $\vp_E$ with eigenvalue
$E$.
\end{theorem}
The reader is referred to \cite {B-R1} for a discussion of this
phenomenon. Our aim here is to prove a quantitative version. Denote
$h_s(A)$ the $s$-dimensional  Hausdorff measure of the set $A$.

\begin{theorem}\label{Theorem3x}

Let $\Sigma$ be as above, $E>E_\Sigma $ and $\vp_E$ an eigenfunction
of $\mathbb T^d$ with eigenvalue $E$. Let $N$ denote the nodal set
of $\vp_E$. Then \be\label{eq9.1} h_{d-2}(N\cap\Sigma)< C_\Sigma
\sqrt E \ee
\end{theorem}

Recall at this point also the Donnelly-Fefferman theorem, stating
that if $M$ is a real-analytic $d$-dimensional manifold and $\vp_E$
an eigenfunction \hfill\break $-\Delta{\vp}= E\vp, \Delta$ the
Lalacian of $M$, then the nodal set $N$ of $\vp_E$ satisfies
\be\label{eq9.2} h_{d-1}(N)< C\sqrt E \ee where $C=C(M)$. See \cite
{D-F}.

As in \cite{D-F}, we will establish \eqref{eq9.1} combining Jensen's
inequality and Crofton's formula. Of course, an additional
ingredient is needed, namely some type of lower bound on the
restriction $\vp|_\Sigma$.

First recall some basic facts on one-variable analytic fuctions.

\begin{lemma}\label{Lemma9.3}
Let $f$ be a bounded analytic function on the unit disc
$D=\{|z|<1\}$. Let $a\in D_{\frac 12} =\{|z|<1\}$ such that
$f(a)\neq 0$ and denote $\nu(D_{\frac 12})$ the number of zeros of
$f$ on $D_{\frac 12}$. Then \be \label{eq9.4} \nu(D_{\frac 12})\leq
C(\big|\log|f(a)|\big|+\log[\sup_{z\in D}|f(z)|]). \ee
\end{lemma}

Hence

\begin{lemma}\label{Lemma9.5}
Let $f\not=0$ be a real analytic function on $\big[-\frac 12, \frac
12]$ with bounded analytic extension to $D$. Let $\nu$ be the number
of zeros of $f$. Then \be\label{eq9.6} \nu\leq C\Big(
\min_{x\in[-\frac 12 \frac 12]}\Big|\log |f(x)|\Big|
+\log\Big[\sup_{z\in D}|\tilde f(z)\Big]\Big). \ee

\medskip
\end{lemma}

\begin{lemma}\label{Lemma9.7}
Let $f$ be as in Lemma \ref{Lemma9.5}. Then \be\label{eq9.8}
\int^{\frac 12}_{-\frac 12} \big|\log |f(x)|\big|dx \leq C
\min_{a\in D_{\frac 12}} \big|\log|\tilde f(a)|\big|+\log
\Big[\sup_{z\in D} |\tilde f(z)|+1\Big]. \ee
\end{lemma}

Lemma \ref{Lemma9.3} follows from Jensen's theorem and \eqref{eq9.8}
is easily deduced from subharmonicity of $\log |\tilde f(z)|$.

Lemma \ref {Lemma9.7} generalizes to real analytic functions of
several variables.

\medskip
\begin{lemma}\label{Lemma9.9}
Let $f\not= 0$ be a real analytic function on $[-\frac 12, \frac
12]^m$, $m\geq 1$ with bounded analytic extension $\tilde f$ to the
polydisc $D^m$. Denote
$$
M=\sup_{z\in D^m} |\tilde f(z)|+1.
$$
Then \be\label{eq9.10} \int_{[-\frac 12, \frac 12]^m} \big|\log |
f(x)|\big|dx\leq C\Big(\min_{a\in D^m_{\frac 12}} \big|\log|\tilde
f(a)|\big|+\log M\Big). \ee
\end{lemma}

\noindent {\sl Proof.} Fix $a\in D^m_{\frac 12}$. From \eqref{eq9.8}
applied to the function $\tilde f(\cdot, a_2, \ldots, a_m)$, we get
\be\label{eq9.11} \int^{\frac 12}_{-\frac12} \Big|\log |\tilde
f(x_1, a_2, \ldots, a_m)|\Big|dx_1 \leq C\big|\log |\tilde
f(a)|\big| +C\log M. \ee Next, fix $|x_1|<\frac 12$ and apply
\eqref{eq9.8} to the function $\tilde f(x_1, ., a_3, \ldots, a_m)$.
Hence \be\label{eq9.12} \int_{-\frac12}^{\frac 12} \big|\log|\tilde
f(x_1, x_2,a_3, \ldots, a_m)|\big| dx_2 \leq C\big|\log|\tilde
f(x_1, a_2, \ldots, a_m)|\big| +C\log M. \ee Integrating
\eqref{eq9.12} in $x_1$ and using \eqref{eq9.11} gives \be
\label{eq9.13} \int^{\frac 12}_{-\frac 12} \int^{\frac 12}_{-\frac
12} \big|\log|\tilde f(x_1, x_2, a_3, \ldots, a_m)|\big| dx_1
dx_2\leq C\big|\log| \tilde f(a)|\big|+ C\log M. \ee Iteration
yields \eqref{eq9.10}.

\begin{lemma}\label{Lemma9.14}
Let $f$ be as in Lemma \ref{Lemma9.9} and
$$Z=\{x\in\big[-\frac 12,  \frac 12]^m; f(x)=0\}$$
Then \be\label{eq9.15} h_{m-1}(Z) \leq C\big(\min_{a\in D^m_{\frac
12}} \big|\log|\tilde f (a)|\big|+\log M\big) \ee
\end{lemma}

\begin{proof}
For $m=1$, \eqref{eq9.15} follows from \eqref{eq9.4}.

For $m>1$, we use Crofton's formula \be\label{eq9.16} h_{m-1}(Z)
\sim \int_{\mathcal L} [\#(Z\cap\ell)]d\ell \ee where $\mathcal L
\simeq G_{m, 1}\times\mathbb R^m$ is the space of affine straight
lines $\ell$.

Fix $\ell\in\mathcal L, \ell\cap Z\not=\phi$ and let $\ell
=b+\mathbb R\xi$, $b\in [-\frac 12, \frac 12]^m$, $|\xi|=1$. Denote
$I$ the interval $I=\{x\in\mathbb R; b+x\xi\in[-\frac 34, \frac
34]^m\}$. Let
$$
g(x) =f(b+x\xi)
$$
which is real analytic with analytic extension $\tilde g$ to
$\{z\in\mathbb C; \dist (z, I)<\frac 12\}$ bounded by $\log M$.
Lemma \ref{Lemma9.5} implies \be\label{eq9.17}
\begin{aligned}
\#[Z\cap\ell]&\leq \#\{ x\in I; g(x)=0\}\\
&\leq c\min_{x\in I} \big|\log |g(x)|\big|+C\log M\\
&\leq c\int_{\ell \cap [-\frac 34, \frac
34]^m}\big|\log|f(x)|\big|+C\log M.
\end{aligned}
\ee Integration of \eqref{eq9.17} over $\mathcal L$ and invoking
Lemma \ref{Lemma9.9} gives
$$
h_{m-1} (Z) \leq c\int_{[-\frac 34, \frac 34]^m}
\big|\log|f(x)|\big|+C\log M\leq\eqref{eq9.15}
$$
proving Lemma \ref{Lemma9.14}.
\end{proof}

\noindent {\bf Proof of Theorem \ref{Theorem3x}}

Let $p:Q\to\Sigma$ be a real analytic parametrization of $\Sigma$
where $Q$ is an $(d-1)$-dimensional interval. Thus
$$
N\cap\Sigma =p\big(\{x\in Q; \vp\big(p(x)\big) =0\}\big).
$$
Since $\vp(x) =\vp_E(x)
=\operatornamewithlimits\Sigma\limits_{\xi\in\mathbb Z^d, |\xi|^2=E}
\hat\vp(\xi) e(x.\xi)$, the analytic extension
$$
\tilde\vp (z_1, \ldots, z_d)= \sum\hat\vp (\xi) e(z_1.\xi_1+\cdots+
z_d\xi_d)
$$
of $\vp$ to the polydisc $D_2^d$ obviously admits a bound \be
|\tilde\vp|< E^{\frac d2} e^{2\sqrt d \sqrt E} =M \ee (assuming
$\Vert\vp\Vert_2=1$). Thus $\vp\circ p$ has an analytic extension to
a complex neighborhood of $Q$, bounded by $M$.

From Lemma \ref {Lemma9.14} \be \label{eq9.19}
\begin{aligned}
h_{d-2} (N\cap\Sigma)&\sim h_{d-2} [x\in Q; \vp\big(p(x)\big)=0]\\
&\leq c\min_{a\in\tilde Q}\big|\log |(\vp\circ p)^\sim
(a)|\big|+c\sqrt E
\end{aligned}
\ee where $\tilde Q\subset\mathbb C^{d-1}$ is some complex
neighborhood of $Q$.

For $d=2$ or $d=3$, our restriction theorem (lower bounds), assuming
$E>E_\Sigma$, implies \be\label{eq9.20} \max_{x\in\Sigma}
|\vp(x)|\gtrsim \Big(\int_\Sigma |\vp|^2 d\sigma\Big)^{\frac 12}>
c_\Sigma \ee and therefore
$$
\log E\gtrsim \log |\vp\big(p(a)\big|>-c
$$
for some $a\in Q$.

For general dimension $d$, we do not have at this point a lower
bound of the type \eqref {eq9.20}. However, the proof of the result
in \cite{B-R1} cited in the beginning of this section, which uses
the complexification $(\vp\circ p)^\sim$, implies in fact that for
$E>E_\Sigma$ \be\label{eq9.21} \max_{a\in\tilde Q} |(\vp\circ
p)^\sim (a)|> E^{-C} \ee where $C$ is some constant. Hence
\eqref{eq9.19} can be applied to obtain \eqref{eq9.1}.  This proves
Theorem \ref{Theorem3x}.

\noindent {\bf Remark.} Theorem \ref{Theorem3x} should be compared
with results in \cite{T-Z} $(d=2)$. Using the \cite{T-Z}
terminology, a real analytic hypersurface $\Sigma \subset \mathbb
T^d$ with nowhere vanishing curvature is `good' in the since that
\be\label{eq9.22} \frac{\max|\vp|\Big|_\Sigma}{\Vert\vp\Vert_2} >
c_\Sigma^{-\sqrt E} \ee for $\vp=\vp_E, E> E_\Sigma$.

\medskip
The remainder of this section deals with the converse phenomenon.

We show for $d=2, 3$ that if $\Sigma\subset \mathbb T^d$ is as above
and $E>E_\Sigma$, $\vp =\vp_E$ an eigenfunction with nodal set $N$,
then
$$
N\cap\Sigma\not= \phi.
$$
For $d=2$, there is a more precise statement.

\begin{theorem}\label{Theorem4}
Let $\Sigma\subset\mathbb T^2$ be a real analytic curve which is not
geodesic. Let $E\geq E_\Sigma$ and $\vp_E$ an eigenfunction with
eigenvalue $E$ and nodal set $N$. Then \be\label{eq9.23} c_\ve
E^{\frac 12-\ve} <\# (N\cap\Sigma)< CE^{\frac 12} \  \text { for all
} \  \ve>0. \ee
\end{theorem}
\begin{proof}
The upper bound follows from (the proof of) Theorem~\ref{Theorem3x},
noting that since $\Sigma$ is not a straight line segment, there is
$\Sigma'\subset\Sigma$ with non-vanishing curvature. For the lower
bound, we can replace $\Sigma$ by $\Sigma'$ and proceed as follows.

Fix $\rho =\frac 12 -\ve_0$ and decompose
$$
\Sigma =\bigcup_{\alpha\lesssim E^\rho} \Sigma_\alpha
$$
in arcs $\Sigma_\alpha$ of size $E^{-\rho}$. From the lower bound
$(\Vert\vp\Vert_2=1)$ \be\label{eq9.24} c_\Sigma =\int_\Sigma
|\vp|^2 d\sigma =\sum_\alpha\int_{\Sigma_\alpha} |\vp|^2 d\sigma \ee
and the upper bound $\Vert \vp\Vert_\infty \ll E^\ve$ for all
$\ve>0$, one easily sees that \be\label{eq9.25} \# \mathcal F=\#
\Big\{ \alpha; \int_{\Sigma_\alpha} |\vp|^2 d\sigma
>cE^{-\rho} \Big\}\gg E^{\rho-\ve} \text { for all $\ve>0$}.
\ee For $\alpha\in\mathcal F$ \be\label{eq9.26}
\int_{\sum_\alpha}|\vp| d\sigma > c\frac {E^{-\rho}}{\Vert
\vp\Vert_\infty}. \ee Hence, if \be\label{eq9.27}
\int_{\Sigma_\alpha} \vp d\sigma
=o\Big(\frac{E^{-\rho}}{\Vert\vp\Vert_\infty}\Big) \ee we can
conclude that $N\cap \Sigma_\alpha\not = \phi$.

Let $\mathcal E =\{\xi\in\mathbb Z^2;|\xi|^2 =E\}$ and $\vp
=\Sigma_{\xi\in\mathcal E}\hat\vp (\xi) e(x\cdot \xi)$,
$\Vert\vp\Vert_2=1$.

Fix $\ve_1>0$ a small number and define \be\label{eq9.28} \mathcal
F_1=\{\alpha\in\mathcal F; \min_{\xi\in\mathcal
E}|\overset{_\rightarrow} t.\xi|> E^{\frac 12-\ve_1} \text { for all
tangent vectors } \overset{_\rightarrow}t  \text{ of $\Sigma_\alpha
\Big\}$}. \ee Clearly
$$
\#(\mathcal F\backslash \mathcal F_1)\lesssim (\# \mathcal E)\
\frac{E^{-\ve_1}}{E^{-\rho}} < E^{\rho-\frac{\ve_1}2}
$$
and \be\label{eq9.29} \# \mathcal F_1 >\frac 12(\#\mathcal F)\gg
E^{\rho-\ve} \ee by \eqref{eq9.25}. Next, letting $\gamma:I=[0,
E^{-\rho}]\rightarrow \Sigma_\alpha$ be an arclength parametrization
of $\Sigma_\alpha, \alpha\in\mathcal F_1$, write
$$
\Big|\int_{\Sigma_\alpha} \vp d\sigma\Big| \leq \sum_{\xi\in\mathcal
  E} |\hat\vp (\xi)|\Big|\int_I e\big(\xi\cdot \gamma(t)\big)dt\Big|
$$
and by partial integration
$$
\Big|\int_I e \big(\xi\cdot \gamma(t)\big) dt\Big|\lesssim
\max_{t\in
  I} \ \frac 1{|\xi \cdot \dot\gamma (t)|} +\int_I \frac{|\xi\cdot \overset
{..}\gamma(t)|}{|\xi \cdot   \dot\gamma (t)|^2} dt\lesssim \frac
         {E^{\ve_1}}{\sqrt E}
$$
 from the definition of $\mathcal F_1$. Hence, for $\alpha\in\mathcal F_1$
\be\label{eq9.30} \Big|\int_{\Sigma_\alpha} \vp d\sigma\Big| \ll
\frac {E^{\ve_1+\ve}}{\sqrt E} \text { for all $\ve>0$} \ee and
\eqref{eq9.27} will hold if $\ve_0>\ve_1$ and $E$ large enough.

It follows that for $E>E_{\Sigma, \ve_0}$
$$
\#(N\cap\Sigma)\geq (\# \mathcal F_1)> E^{\frac 12-2\ve_0}
$$
proving Theorem \ref{Theorem4}.
\end{proof}

\bigskip

For $d=3$, we can show

\begin{theorem}\label{Theorem5}
Let $\Sigma\subset\mathbb T^3$ be a real analytic surface with
non-vanishing curvature. There is $E_\Sigma$ such that if
$E>E_\Sigma$, $E\neq 0,4,7\bmod 8$ and $N$ is the nodal set of
$\vp_E$, then \be\label{eq9.31} N\cap \Sigma\not= \phi\;. \ee
\end{theorem}

The argument allows more precise statements that we do not attempt
to formulate here.

As before, \eqref{eq9.31} will be derived from a property
\be\label{eq9.32} \Big|\int_\Sigma \vp(x) \omega(x) d\sigma\Big|=
o\Big(\int_\Sigma |\vp(x) \, \omega(x) | d\sigma\Big) \ee with
$0\leq \omega\leq 1$ a smooth localizing function on $\Sigma$.

Letting
$$
\phi(x) =\sum_{\xi \in\mathcal E} \hat\phi (\xi) e(x \cdot\xi)
\qquad\mathcal E=\{\xi\in\mathbb Z^3; |\xi|^2 =E\}
$$
Then \eqref{eq1.10} allows to bound the left side of \eqref{eq9.32}
by $(\Vert\phi\Vert_2=1)$ \be\label{eq9.33}
\begin{aligned}
\sum_{\xi\in\mathcal E} |\hat\phi(\xi)| \ \Big|\int_\Sigma e(x\cdot
\xi) \omega(x) d\sigma\Big| &\lesssim \sum_{\xi\in\mathcal E}
\frac{|\hat\phi(\xi)|}{\sqrt E}
\\
&\lesssim E^{-\frac 12}(\#\mathcal E)^{\frac 12}\ll E^{-\frac
14+\ve}.
\end{aligned}
\ee According to Theorem \ref{Theorem2x}, \be \int_\Sigma |\vp|^2
\omega d\sigma >c \ee and hence, certainly
$$
\int_\Sigma |\vp|\omega d\sigma>\frac c{\Vert\vp\Vert_\infty}>
\Big(\sum |\hat\vp(\xi)|\Big)^{-1}\gg E^{-\frac 14-\ve}
$$
which is barely insufficient to conclude.

Instead of interpolating $L^2(\Sigma, d\sigma)$ between $L^1(\Sigma,
d\sigma)$ and $L^{\infty} (\Sigma, d\sigma)$, interpolate
$L^2(\Sigma, d\sigma)$ between $L^1(\Sigma, d\sigma)$ and
$L^4(\Sigma, d\sigma)$ \be\label{eq9.35} c<\int_\Sigma |\vp|^2\omega
d\sigma \leq \Big(\int|\vp|\omega d\sigma\Big)^{\frac 23}
\Big(\int|\vp|^4\omega d\sigma\Big)^{\frac 13} \ee reducing to
problem to establish a bound of the form \be\label{eq9.36}
\int_\Sigma |\vp|^4 \omega d\sigma < E^{\frac 12-\ve_0} \ee for some
$\ve_0>0$.

Note that from Theorem \ref{Theorem1} \be\label{eq9.37} \int_\Sigma
|\vp|^2 \omega d\sigma < C \ee and therefore \be\label{eq9.38}
\int_\Sigma |\vp|^4 \omega d\sigma <C\Vert\vp\Vert^2_\infty \leq
C\Big(\sum |\hat\vp(\xi) |\Big)^2\ll E^{\frac 12+\ve}. \ee
Decomposing $\vp =\vp_1 +\vp_2$, with
$$
\vp_1(x) =\sum_{|\hat\vp(\xi)| > E^{-\frac 14+\ve_1}} \hat\vp (\xi)
e(x\cdot \xi)
$$
the bound \eqref{eq9.38} implies that
$$
\int_\Sigma |\vp_1|^4 \omega d\sigma < E^{\frac 12-2\ve_1}
$$
and hence we may assume \be\label{eq9.39} |\hat\vp(\xi)|< E^{-\frac
14+}. \ee Fix $\rho =\frac 12-\tau$, $\tau>0$ sufficiently small,
and partition
$$
\mathcal E =\bigcup_\alpha \mathcal E_\alpha
$$
in $\sim E^{2\tau} $ sets of diameter at most $E^\rho$. Write
$$
\vp =\sum_\alpha \vp_\alpha\text { with } \vp_\alpha (x)
=\sum_{\xi\in \mathcal E_\alpha} \hat\vp (\xi) e(x\cdot \xi)
$$
and \be\label{eq9.40} \int_\Sigma |\vp|^4 \omega d\sigma\leq
E^{2\tau} \sum_\alpha \int_\Sigma |\vp_\alpha|^2 |\vp|^2\omega
d\sigma\;. \ee We choose $\tau$ small enough for Linnik's
equidistribution property to imply
\be\label{eq9.41} \#\mathcal E_\alpha\ll E^{\frac 12-2\tau+o(1)}  \
\text { for each
  $\alpha$}
\ee (this is where we need to assume $E\neq 0,4,7\bmod 8$).
Expanding in Fourier and using again \eqref {eq1.10}, we obtain \be
\label{eq9.42}
\begin{aligned}
&\int_\Sigma |\vp|^2 |\vp_\alpha|^2 \omega d\sigma\lesssim\\
&\Big|\sum_{\xi_1-\xi_2+\xi_3-\xi_4\not= 0} \hat\vp(\xi_1)
\overline{\hat\vp(\xi_2)} \ \hat\vp_\alpha(\xi_3)
\overline{\hat\vp_\alpha(\xi_4)} \ \frac {e^{i\psi(\xi_1-\xi_2+
\xi_3-\xi_4)}}{|\xi_1 -\xi_2+\xi_3-\xi_4|}\Big|
\end{aligned}
\ee \be \label{eq9.43} +\sum_{\xi_1, \xi_2, \xi_3, \xi_4} \
|\hat\vp(\xi_1)| \, |\hat\vp(\xi_2)| \, |\hat\vp_\alpha(\xi_3)| \,
|\hat\vp_\alpha(\xi_4)| \ (1+|\xi_1-\xi_2+\xi_3+\xi_4|)^{-2}. \ee
From \eqref {eq9.39}, clearly \be \label{eq9.44}
\begin{aligned}
\eqref{eq9.43}&< E^{-1+o(1)} \sum_{\xi_1, \xi_2 \in\mathcal E; \xi_3, \xi_4 \in\mathcal E_\alpha} (1+|\xi_1-\xi_2+\xi_3-\xi_4|)^{-2}\\
&\lesssim E^{-1+o(1)} \sum_{2^k< E^{\frac 12}} 4^{-k} \# \{(\xi_1,
\xi_2, \xi_3, \xi_4)\in\mathcal E^2 \times\mathcal E^2_\alpha;
|\xi_1-\xi_2+\xi_3-\xi_4|< 2^k\}.
\end{aligned}
\ee We need to estimate for $r<R$ \be \label{eq9.45} \#\{(\xi_1,
\xi_2, \xi_3, \xi_4)\in \mathcal E^2\times\mathcal E_\alpha^2;
|\xi_1-\xi_2+\xi_3-\xi_4|< r\}. \ee Let $P_\delta$ be an approximate
identity on $\mathbb T^3$. Then
\begin{equation}\label{eq9.46}
\begin{aligned}
\eqref{eq9.45}&\lesssim \int_{\mathbb T^3}\Big|\sum_{\xi\in\mathcal
E} e(x.\xi)\Big|^2 \ \Big|\sum_{\xi\in \mathcal E_\alpha}
e(x.\xi)\Big|^2 P_{\frac 1r}(x)dx\\
&\lesssim r^3 \int_{\mathbb T^3} \Big|\sum _{\xi\in\mathcal E} e(x.\xi)\Big|^2 \ \Big|\sum_{\xi\in\mathcal E_\alpha} e(x.\xi)\Big|^2 dx\\
&=r^3\# \{(\xi_1, \xi_2, \xi_3, \xi_4)\in\mathcal E^2\times\mathcal E^2_\alpha; \xi_1-\xi_2 =\xi_3-\xi_4\}\\
&\leq r^3 |\mathcal E_\alpha|^2\Big[\max_{v\in\mathbb Z^3} (\#\{(\xi, \eta)\in \mathcal E^2; \xi-\eta =v\})\Big]\\
&\ll r^3 |\mathcal E_\alpha|^2 E^\ve\\
\end{aligned}
\ee \be\label{eq9.47} \ll r^3 E^{1-4\tau+\ve}. \
\quad\qquad\qquad\qquad\qquad\quad\qquad\qquad \ee We used here
\eqref{eq9.47} and the bound \be\label{eq9.48} \# \{(\xi,
\eta)\in\mathbb Z^3 \times \mathbb Z^3: |\xi|^2 =E=|\eta|^2 \ \text
{ and } \ \xi-\eta =v\}\ll E^\ve \ee which is a consequence of
\eqref{eq8.7}.

Another bound on \eqref{eq9.43} is obtained by fixing
$\xi_2\in\mathcal E, \xi_3, \xi_4 \in\mathcal E_\alpha$ and
observing that $\xi_1\in \mathcal E$ is restricted to some ball of
radius $r$. Hence, invoking Lemma \ref{Lemma6.8},
$$
\eqref{eq9.45}\ll |\mathcal E_\alpha|^2|\mathcal E|E^\ve \Big( 1+r\Big(\frac r{\sqrt E}\Big)^{\frac 1{20}}\Big)\\
$$
\be\label{eq9.49} \ll E^{\frac 32-4\tau+\ve}\Big(1+r\Big(\frac
r{\sqrt E}\Big)^{\frac 1{20}}\Big). \ee Thus \be\label{eq9.50}
\eqref{eq9.43} \ll E^{-4\tau+\ve} \sum_{2^k< E^{\frac 12}} \min
(2^k, 4^{-k} E^{\frac 12}+ 2^{-k} E^{\frac 12} \Big(\frac{2^k}{\sqrt
E}\Big)^{\frac 1{20}}\Big) \ll E^{\frac 14-4\tau+\ve}. \ee

Next, we estimate \eqref {eq9.42}. Let $0\leq\eta\leq 1$ be a bump
function on $\mathbb R^3$ such that $\eta(x)=0$ if $|x|<\frac 12$ or
$|x|\geq 2$.  Estimate
\begin{multline}\label{eq9.51}
\eqref{eq9.42}<\sum_{2^k<\sqrt E}\Big|\sum_{\xi_1, \xi_2, \xi_3,
  \xi_4} \eta\Big(\frac {\xi_1-\xi_2+\xi_3-\xi_4)}{2^k}\Big)
\hat\vp(\xi_1) \overline{\hat\vp (\xi_2)}
\hat\vp_\alpha (\xi_3) \overline{\hat\vp_\alpha(\xi_4)} \\
 \times \frac{e^{i\psi(\xi_1- \xi_2+\xi_3-\xi_4)}}{|\xi_1-\xi_2+\xi_3-\xi_4|}\Big|.
\end{multline}
Ignoring the oscillatory factor, the $k$-term in \eqref{eq9.51} can
be estimated by \be\label{eq9.52} E^{-1+o(1)} 2^{-k}[\# \{(\xi_1,
\xi_2, \xi_3, \xi_4)\in \mathcal E^2\times\mathcal E_\alpha^2;
|\xi_1-\xi_2+\xi_3-\xi_4|\lesssim 2^k\}] \ee recalling
\eqref{eq9.39}.  From \eqref{eq9.47}, \eqref{eq9.49}
$$
\eqref{eq9.52}\ll E^{-4\tau+\ve} \min (4^k, 2^{-k} E^{\frac 12}+E^{\frac 12}\Big(\frac{2^k}{\sqrt E}\Big)^{\frac 1{20}}\Big)\\
$$
\be\label{eq9.53} \ll E^{\frac 13 -4\tau+\ve}+E^{\frac
12-4\tau+\ve}\Big(\frac{2^k}{\sqrt E}\Big)^{\frac 1{20}}. \ee This
estimate is conclusive unless $2^k> E^{\frac 12-\ve_1}\gg E^\rho$
($\ve_1>0$ an arbitrary small fixed constant). For such $k$, the
oscillatory factor in \eqref{eq9.51} cannot be ignored.

Estimate the $k$-term in \eqref{eq9.51} by
\begin{multline}\label{eq9.54}
(\#\mathcal E_\alpha)^2 \cdot E^{-\frac 12+\ve} \max_{\xi_3,
  \xi_4\in\mathcal E_\alpha}\Big|\sum_{\xi_1, \xi_2\in\mathcal E} \eta
\Big(\frac{\xi_1-\xi_2+\xi_3-\xi_4}{2^k}\Big)
\hat\vp(\xi_1) \overline{\hat\vp (\xi_2)} \\
\times \frac{e^{i\psi(\xi_1-\xi_2+\xi_3-\xi_4)}}{|\xi_1-\xi_2+\xi_3-\xi_4|}\Big| \\
\ll 2^{-k} E^{-4\tau+\ve}\Big\{ \max_{\substack{|v|<E^\rho\\
|a_\xi|,
    |b_\xi|\leq 1}}\Big|\sum_{\substack{\xi_1, \xi_2\in\mathcal E\\
|\xi_1-\xi_2|>2^{k-2}}} a_{\xi_1}b_{\xi_2}  \
e^{i\psi(\xi_1-\xi_2+v)}\Big|\Big\}.
\end{multline}

In establishing \eqref{eq9.36}, we may obviously assume $\diam(\supp
\hat\vp)<c\sqrt E$ ($c$ as in Lemma~\ref{Lemma5.11}). It remains to
get a nontrivial bound on \be\label{eq9.55}
\sum_{\substack{\xi_1, \xi_2\in \mathcal E\\
2^{k-2} <|\xi_1-\xi_2| < c\sqrt E\\ \xi_1-\xi_2 +v\in Z}} a_{\xi_1}
b_{\xi_2} \, e^{i\psi(\xi_1-\xi_2+v)} \ee where $|v|< E^\rho,
\rho<\frac 12-\ve_1$. The same analysis used to prove Lemma
\ref{Lemma5.11} gives an estimate \be\label{eq9.56}
|\eqref{eq9.55}|< E^{1-\gamma} \ee for some $\gamma>0$\footnote
{Letting $R=\sqrt E$, $v'=\frac vR, |v'|<R^{-2\tau}$, consider the
function $\psi(x-y+v')$ with $x, y\in S^2$. The sets $\mathcal
W_{\delta, \delta_1}$ considered in Lemma \ref{Lemma2.18} for the
function $\psi\big(p(x) -p(y)\big)$ remain the same for the function
$\psi\big(p(x) - p(y)+v'\big)$, since $\delta, \delta_1> R^{-\ve}$
while $|v'|<R^{-2\tau}$, $\ve<\tau$. Thus the analysis from Section
4 still applies and we obtain Lemma~\ref{Lemma3.35} for $\psi(x-y)$
replaced by $\psi(x-y+v)$.}

Hence \be\label{eq9.57} \eqref{eq9.54} \ll E^{\frac
12-\gamma-4\tau+\ve_1+\ve} < E^{\frac 12-\frac 12\gamma-4\tau}. \ee
Thus from \eqref{eq9.53}, \eqref {eq9.57}, we obtain
\be\label{eq9.58} \eqref{eq9.42}\ll E^{\frac 12 -4\tau-\frac 1{20}
\ve_1+\ve}+ E^{\frac 12-\frac 12\gamma-4\tau} \ee and recalling
\eqref{eq9.40} \be\label{eq9.59} \int_\Sigma |\vp|^4 d\sigma
\lesssim E^{4\tau} \times \eqref{eq9.58} < E^{\frac 12-\frac
1{21}\ve_1} +E^{\frac 12-\frac 12\gamma} \ee

This completes the proof of \eqref{eq9.36} and
Theorem~\ref{Theorem5}. \qed

\noindent {\bf Remark.}

It is easily seen that if $\Sigma\subset\mathbb T^3$ is a smooth
surface, then \be\label{eq9.60} \max_{\Vert\vp_E\Vert_2 =1}
\Big(\int_\Sigma |\vp_E|^4 d\sigma\Big) \gtrsim \frac 1E
(\#\{\xi\in\mathbb Z^3; |\xi|^2 =E\})^2 \ee (consider the
contribution of $\Sigma'\cap B(0, \frac 1{10R})$ with $0\in\Sigma'$
a shift of $\Sigma$).

Since by \eqref {eq8.6} there are arbitrary large eigenvalues $E$
for which \be \#\{\xi\in\mathbb Z^3; |\xi|^2 =E\}\gtrsim E^{1/2}
(\log\log E) \ee one can not hope for uniform $L^4$-restriction
bounds.

\section{Higher Dimension}

\subsection{}
 We are not able at the time of this writing to prove either
Theorem~ \ref{Theorem1} or Theorem~\ref{Theorem2x} in dimension
$\geq 4$.

It was proven by R.~Hu \cite{H} that if $(M, g)$ is a smooth compact
Riemannian manifold of dimension $d$ and $\Sigma$ a smooth
submanifold of dimension $d-1$ with positive (or negative) definite
second fundamental form, then \be\label{eq10.1}
\Vert\vp_E\Vert_{L^2(\Sigma)} < C_\Sigma E^{\frac 1{12}}
\Vert\vp_E\Vert_{L^2(M)} \ee for all eigenfunctions $\vp_E,
-\Delta\vp_E =E\vp_E$ of the Laplace-Beltrami operator $\Delta$ of
$M$. For $d=2$, the result is due to \cite{BGT}.

In the case of the flat torus $M=\mathbb T^d$, one can show an
improvement over \eqref{eq10.1} in arbitrary dimension
\be\label{eq10.2} \Vert \vp_E\Vert_{L^2(\Sigma)} < C_\Sigma E^{\frac
1{12}-\ve_d}\Vert\vp_E\Vert_{L^2(\mathbb T^d)} \ee for some
$\ve_d>0$ (with same assumption on $\Sigma$).

We will not present the proof here, as we believe the validity of
our Theorem \ref{Theorem1} in any dimension is the truth.

\subsection{}
Theorem \ref{Theorem1} in its dual formulation is the following
statement about restriction of the Fourier transform.

\begin{theorem}\label{Theorem1x}
Let $\Sigma \subset\mathbb T^3$ be real analytic with nowhere
vanishing curvature. For $E\in\mathbb Z_+$, denote
$$
\mathcal E_E=\{\xi\in\mathbb Z^3; |\xi|^2=E\}.
$$
Then the restriction operator
$$
L^2\Big(\Sigma, d\sigma\Big) \rightarrow \ell^2(\mathcal E_E) :
\mu\rightarrow \hat\mu|_{\mathcal E_E}
$$
has norm bounded by $C_\Sigma$.
\end{theorem}
Setting $R=\sqrt E$, our argument involves the following properties
of
 $\mathcal E=\mathcal E_E\subset RS^2$:

\noindent{\bf i)} There is $\ve_1>0$ such that if $r=R^{1-\ve_1}$
and $C_r\subset RS^2$ is a cap of size $r$, then (for some
sufficiently small $\ve>0$) \be\label{eq10.3} |\mathcal E\cap
C_r|\lesssim \Big(\frac rR\Big)^2 R^{1+\ve} \ee

\noindent{\bf ii)} There is some constant $\eta>0$ such that if
$r<R$ and $C_r\subset RS^2$, then \be\label{eqref10.4} |\mathcal
E\cap C_r|\ll R^\ve\Big[\Big(\frac rR\Big)^\eta r+1\Big] \text { for
all $\ve>0$} \ee

\noindent{\bf iii)}  Denoting $\{C_\alpha\}$ a partition of $RS^2$
in cells of size $\sim\sqrt R$, \be\label{10.5} \sum_\alpha|\mathcal
E\cap C_\alpha|^2 \ll R^{1+\ve} \text { for all $\ve>0$} \ee holds.

 Note that we did not use the fact that $\mathcal E\subset\mathbb Z^3$.

The `idealization' of $\mathcal E$ is a set $\mathcal S\subset RS^2$
which elements are $\sqrt R$-separated. For such sets $\mathcal S$,
the restriction operator \be\label{eq10.6} L^2 \Big(\Sigma,
d\sigma\Big)\rightarrow \ell^2 (\mathcal S):\mu\rightarrow
\hat\mu\big|_{\mathcal S} \ee with $\Sigma$ as in Theorem
\ref{Theorem1x}, is easily seen to be bounded. By \eqref{eq1.10}, it
suffices indeed to show that \be\label{eq10.7} \sum_{\xi,
\xi'\in\mathcal S} \frac{|a_\xi| \, |a_{\xi'}|} {|\xi-\xi'|+1} \leq
C\big(\sum_{\xi\in\mathcal S}|a_\xi|^2\Big). \ee Our assumption on
$\mathcal S$ implies that $\max_{\xi'} \big(\Sigma_{\xi\in S} \frac
1{|\xi-\xi'|}\big) < C$ and \eqref {eq10.7} follows from Schur's
test.

Surprisingly, the higher dimensional analogue, where one considers a
set $\mathcal S\subset RS^{d-1}$ of $R^{\frac 1{d-1}}$-separated
points as idealization of $\mathcal E =\{\xi\in\mathbb Z^d; |\xi|^2
=R^2\}$, may fail for $d$ large enough. This illustrates the
difficulty of proving  Theorem~\ref{Theorem1} for general dimension
and the need to exploit somehow that $\mathcal E\subset\mathbb Z^d$.

In the next example $\Sigma =S^{d-1}$.

\begin{lemma} \label{Lemma10.8}

Let $d\geq 8$. Then for large $R$ there is a set $S=S(R)\subset
\hfill\break \{x \in\mathbb R^d| \, |x|=R\}$ with the following
property \be\label{eq10.9} |\xi-\xi'| \gtrsim R^{\frac 1{d-1}} \text
{ for } \xi\not= \xi' \text { in } S \ee and such that the operator
$$
L^2(S^{d-1}, d\sigma)\to \ell^2(S): \mu\mapsto\hat\mu|S
$$
has norm at least $R^{\frac 16- \frac 1{d-1}}$.
\end{lemma}

\begin{proof}
Let $K=[R^{\frac 1{d-1}}]$. In fact we will only use points in the
cap
$$
C=\{|x|=R\} \cap B\Big(Re_d, \frac 1{100} R^{2/3}\Big).
$$
(see Figure~\ref{figcirc}). We choose \be\label{eq10.10} S=\Big\{
\Big(Kz_1, \ldots, Kz_{d-1},
\sqrt{R^2-K^2(z_1^2+\cdots+z^2_{d-1})}\Big), z_i\in\mathbb Z, |z|<
\frac {R^{2/3}}{100K}\Big\} \ee

\begin{figure}[h]
\begin{center}
\input circ.tex
\caption{} \label{figcirc}
\end{center}
\end{figure}

Next, we introduce the measure $\mu$ on $S_{d-1}$,
$\Vert\frac{d\mu}{d\sigma}\Vert_2 =1$. Let
$$
\mathcal F=\{(y_1, \ldots, y_{d-1})\in \mathbb Z^{d-1};
y^2_1+\cdots+ y^2_{d-1} =K^2\} \;
$$
Thus
$$
|\mathcal F| \sim K^{d-3}.
$$
Define
$$
\Omega =\Big\{x=(x_1, \ldots, x_d)\in S^{d-1}; \dist \Big(x', \frac
1K\mathcal F\Big)< R^{-\frac 23}\Big\}
$$
where $ x'=(x_1, \ldots, x_{d-1})$. Hence
$$
1-x^2_d =|x'|^2 >1- 2R^{-2/3} \text { and } |x_d|< \sqrt 2 R^{-\frac
13}.
$$
Also \be\label{eq10.11} |\Omega| \sim |\mathcal F| \cdot R^{-\frac
23(d-2)} R^{-\frac 13} \sim K^{d-3} R^{-\frac 23 d+1}. \ee Define
$\mu$ on $S_{d-1}$ by \be\label{eq10.12} \frac{d\mu}{d\sigma} =
\frac{e(-R\cdot x_d)}{|\Omega|^{\frac 12}} 1_\Omega\;. \ee Evaluate
\be\label{eq 10.13} \sum_{\xi\in S} |\hat\mu(\xi)|^2 =|\Omega|^{-1}
\sum_{\xi\in S} |\widehat {1_\Omega} (\xi -Re_d)|^2. \ee Note that
$S-Re_d$ is contained in $\frac 1{100} R^{2/3}\times\cdots\times
\frac 1{100} R^{2/3}\times \frac 1{100} R^{1/3}$ and therefore, from
definition of $\Omega$ and $\mathcal S$ \be\label{eq10.14}
\begin{aligned}
e\big((\xi -Re_d)\cdot x\big) &\approx e(\xi_1x_1+\cdots+ \xi_{d-1} x_{d-1})\\
&\approx e(Kz \cdot x')=1
\end{aligned}
\ee for $\xi\in S, x\in\Omega$ and $z\in\mathbb Z^{d-1}\cap B\big(0,
\frac{R^{2/3}}{100K}\big)$. It follows from the definition of
$\Omega$ that \be\label{eq10.15} \Vert \hat\mu |_S\Vert^2_2 \gtrsim
|\mathcal S| \, |\Omega| \sim\Big(\frac{R^{2/3}}K\Big)^{d-1} \cdot
K^{d-3} R^{-\frac 23 d+1}=R^{\frac 13} K^{-2} \ee hence  the claim.
Note that we may replace $S$ by $T(S)$, with $T$ an arbitrary
orthogonal transformation of $\mathbb R^d$, with the same conclusion
for the restriction operator. \end{proof}

\section{Restriction upper bounds for generic  eigenvalues}
\label{sec:generic}

In this section we prove Theorem~\ref{Theorem6}. The proof of
Theorem~\ref{Theorem6} is based on the following arithmetic
statement (Lemma 2.9 in \cite{BR2}).

\begin{lemma}\label{Lemma11.1}
Fix $\ve>0$ and taking $N\in\mathbb Z_+$ large, $E\in\{1, \ldots,
N\}$ and $\lambda =\sqrt E$, one has that \be\label{eq11.2}
\min_{\substack {x\not= y\in\mathbb Z^2\\ |x|=\lambda =|y|}} |x-y|>
\lambda^{1-\ve} \ee except for a set of $E$-values of size at most
$N^{1-\frac\ve 3}$.
\end{lemma}

We recall the argument.

\noindent {\bf Proof of Lemma \ref{Lemma11.1}.}

Let $M=\sqrt N$ and estimate the size of the set \be\label{eq11.3}
S=\{x\in\mathbb Z^2; |x|\leq M \text { and } 2x. z=|z|^2 \text { for
some } z\in\mathbb Z^2, 0<|z|<M^{1-\ve}\}. \ee Writing $z=d.z',
d\in\mathbb Z_+$ and $z'=(z_1', z_2')\in\mathbb Z^2$ primitive, the
equation \be\label{eq11.4} 2x.z'=d|z'|^2 \ee has at most $c\frac
M{|z'|}$ solutions in $x, |x|\leq M$, for given $z'$ primitive.

Hence
$$
|S|\leq C\sum_{1\leq d<M} \ \sum_{\substack{z'\in \mathbb Z^2\\
0<|z'|<\frac{M^{1-\ve}} d}} \ \frac M{|z'|} < C\sum_{d< M}
\frac{M^{2-\ve}} d< CM^{2-\ve}\log N.
$$
Since $|S|$ is obviously an upper bound for the number of
exceptional $E\in \{1, \ldots, N\}$, Lemma \ref{Lemma11.1} follows.

\medskip
Theorem \ref{Theorem6} is therefore a consequence of
\medskip

\begin{lemma}\label{Lemma11.4}

Let $\ve>0$ be small enough and $E=\lambda^2 \in\mathbb Z_+$ satisfy
\eqref {eq11.2}.

Let $\Sigma$ be a $C^2$-smooth curve in $\mathbb T^2$. Then any
eigenfunction $\vp_\lambda$ of $\mathbb T^2$ satisfies
\be\label{eq11.5} \Vert\vp_\lambda\Vert_{L^2(\Sigma)} \leq C_\Sigma
\Vert\vp_\lambda\Vert_2.\ee
\end{lemma}

\begin{proof}

Let $\gamma:I\to \Sigma$, $I\subset [0, 1]$, be an arclength
parametrization. Fix $\frac 12<\rho<1$ and partition $I=\bigcup I_s,
I_s =[t_s, t_{s+1}]$ in intervals of size $\lambda^{-\rho}$. Since
$$\gamma(t) =\gamma(t_s)+(t-t_s)\dot\gamma (t_s)+O(\lambda^{-2\rho})$$
for $t\in I_s$, it follows  that \be\label{eq11.6}
\begin{split}
\Big|\int_I e^{i\xi.\gamma(t)}dt\Big| &\leq
\sum_s\Big|\int_{t_s}^{t_{s+1}} e^{i\xi\cdot  \gamma(t)} dt\Big| \\
&= \sum_s\Big|\int_0^{\lambda^{-\rho}} e^{i\xi\cdot  \dot\gamma(t_s)
u}du \Big| + O(|\xi|\lambda^{-2\rho}).
\end{split}
\ee Denote $\mathcal E =\{\xi\in\mathbb Z^2; |\xi|=\lambda\}$ and
$\vp=\sum_{\xi\in\mathcal E} a_\xi e(x\cdot \xi)$,
$\Vert\vp\Vert_2\leq 1$. Estimate using \eqref{eq11.6}
\be\label{eq11.7}
\begin{aligned}
\int_\Sigma |\vp|^2 d\sigma&\leq \sum_{\xi, \xi'\in\mathcal E}|a_\xi|\, |a_{\xi'}|\Big|\int_I e\big(\gamma(t)\cdot (\xi-\xi')\big)dt\Big|\\
&\leq \sum_s \ \sum_{\xi, \xi'\in \mathcal E} |a_\xi| \, |a_{\xi'}|
\min \Big\{ \lambda^{-\rho}, \frac 1{|(\xi-\xi')\cdot
\dot\gamma(t_s)|}\Big\} +|\mathcal E|^2 \lambda^{1-2\rho}.
\end{aligned}
\ee Fix $1\leq s\leq \lambda^\rho$. If we fix $\xi\in\mathcal E$ and
let $\xi'\in\mathcal E\backslash \{\xi\}$ vary, it follows from
\eqref{eq11.2} that \be\label{eq11.8}
|P_{\dot\gamma(t_s)}(\xi-\xi')|\gtrsim
\lambda^{-\ve}|\xi-\xi'|\gtrsim \lambda^{1-2\ve} \ee

\noindent except for at most 1 element. Thus \eqref{eq11.8} holds
for $(\xi, \xi')\not\in \mathcal E_s$ where $\mathcal
E_s\subset\mathcal E$ has the property that for each $\xi$ (resp.
$\xi'$) there is at most one $\xi'$ (resp. $\xi$) with $(\xi,
\xi')\in\mathcal E_s$.  From \eqref{eq11.8}
$$
\begin{aligned}
\eqref{eq11.7}&\lesssim \sum_s \ \sum_{\xi, \xi'\in\mathcal
E}|a_\xi| \, |a_{\xi'}|
\lambda^{-1+2\ve}+\sum_s\lambda^{-\rho}\sum_{(\xi, \xi')\in \mathcal
E_s}
|a_\xi| \, |a_{\xi'}|+|\mathcal E|^2\lambda^{1-2\rho}\\
&\lesssim \lambda^{-1+\rho+2\ve}|\mathcal E|^2 +\max_s \sum_{(\xi, \xi')\in\mathcal E_s}|a_\xi| \, |a_{\xi'}|+|\mathcal E|^2 \lambda^{1-2\rho}\\
&\lesssim 1+ |\mathcal E|^2 (\lambda^{-1+\rho+2\ve}+
\lambda^{1-2\rho})\lesssim 1
\end{aligned}
$$
for $\ve>0$ small enough, since $\frac 12<\rho<1$.

This proves Lemma \ref{Lemma11.4}.
\end{proof}

\section{The number of nodal domains for a random eigenfunction}
\label{sec:NaSod}

In this section  we prove the analogue of the Nazarov-Sodin theorem
\cite{N-S} on the number of nodal domains for $\mathbb T^d, d\geq
3$. We restrict ourselves to $d=3$ as some extra arithmetical
assumptions are required in this case.

\begin{theorem}\label{Theorem7}
Let $d=3$. Assume $E=\lambda^2 \in\mathbb Z$ sufficiently large and
$E\not= 0, 4, 7 (\text{mod}\, 8)$. The number of components of the
nodal set $N$ of a `typical' eigenfunction $\vp_\lambda$
 is of the order $\lambda^3$.
\end{theorem}

\begin{proof}
In \cite{N-S}, the corresponding result is proven for the sphere,
based on a `barrier' argument. It turns out that the same method can
be easily adapted to the torus $\mathbb T^d$, at least when $d\geq
3$, to produce the required lower bound (the upper bound follows
from Courant's nodal domain theorem).

First, denoting $X_\lambda = \text{ span }  \{\vp; -\Delta\vp=
\lambda^2\vp\}$ and $P(X_\lambda)$ the corresponding projective
space,
 a generic element of $P(X_\lambda)$ is represented by a Gaussian random variable
\be\label{11.9} \vp^\omega(x) =\frac 1{|\mathcal E|^{\frac 12}}
\sum_{\xi\in\mathcal
  E} \big(g_\xi(\omega)\cos 2\pi x\cdot \xi+h_\xi (\omega) \sin 2\pi x\cdot
\xi) \ee with $\mathcal E\cap(-\mathcal E)=\phi$, $\mathcal E\cup
(-\mathcal E)= \{\xi\in\mathbb Z^d: |\xi|=\lambda\}$ and $\{g_\xi\},
\{ h_\xi\}$ independent, real, normalized Gaussian random variables.

Denoting \be\label{eq11.10} \{f_j\}=\{\sqrt 2\cos 2\pi x\cdot \xi,
\quad \sqrt 2\sin 2\pi x\cdot \xi : \xi\in\mathcal E\} \ee rewrite
$\vp^\omega$ as \be\label{eq11.11} \vp^\omega = \frac
1{\sqrt{2|\mathcal E|}} \sum_{1\leq j\leq 2|\mathcal E|} g_j(\omega)
f_j \ee where $\{g_j\}$ are as above.

Denote $N= 2|\mathcal E|$ and let $T$ be an $N\times N$ orthogonal
matrix. Defining \be\label{eq11.12} F_i(x) =\sum^N_{j=1} T_{ij}
f_j(x) \ee the Gaussian random variable $\vp^\omega$ has the same
distribution as \be\label{eq11.13} \psi^\omega=\frac 1{\sqrt N}
\sum^N_{j=1} g_j(\omega) F_j \ee (by invariance of the Gaussian
ensemble under the orthogonal group).

Choose $T$ with \be\label{eq11.14}
T_{1j} =\begin{cases} \frac 1{\sqrt{|\mathcal E|}} &\text { if $f_j$ is even}\\
0 &\text { if $f_j$ is odd}.
\end{cases}
\ee Hence \be\label{eq11.15} F_1(x) =\frac{\sqrt 2}{\sqrt{|\mathcal
E|}} \sum_{\xi\in\mathcal E} \cos 2\pi x\cdot  \xi \ee that we use
as our `barrier' function.

Rewrite \be\label{eq11.16} \psi^\omega =\frac 1{\sqrt N} g_1
(\omega) F_1+G^\omega \ee with $G^\omega$ independent of $g_1$.


Taking in \eqref{eq11.15} $\Vert x\Vert \lesssim \lambda^{-1}$, it
follows from the equidistribution of lattice  points on the sphere
(this is why we impose the condition $E\neq 0,4,7 \bmod 8$, see
\S~\ref{sec:9.1}) that \be\label{11.17}
\begin{aligned}
F_1(x) &= \sqrt N\Big\{\int_{S^2} (\cos 2\pi \lambda
x\cdot  \zeta)\sigma(d\zeta)+ O(\lambda^{-\ve})\Big\}\\
& =\sqrt N \big(\hat\sigma(\lambda|x|)+ O(\lambda^{-\ve})\big).
\end{aligned}
\ee Therefore, there is some $r\sim \frac 1\lambda$ such that (for
some constant $c>0$) \be\label{eq11.18} F_1 (x) < -c\sqrt N \text {
for } |x| =r. \ee Also, clearly \be\label{eq11.19} F_1(0) =\sqrt N.
\ee Assume we show that for some constant $C_1$, \be\label{eq11.20}
\max_{|x|\leq r} |G^\omega(x)|<C_1 \ee holds with probability at
least $\frac 12$ in $\omega$.

Since $g_1(\omega)$ is independent of $G^\omega$, it follows from
\eqref{eq11.18}, \eqref{eq11.19} \be\label{eq11.21} \psi^\omega(0)
\geq g_1(\omega)-|G^\omega(0)|> C_2 - C_1> 1 \ee and for $|x|=r$
\be\label{eq11.22} \psi^\omega(x) < -c g_1(\omega)+\max_{|x|=r}
|G^\omega(x)|< -cC_2+C_1< -1 \ee with probability at least $\frac 12
e^{-C_2^2} > c_3>0$ in $\omega$. For such $\omega$, since
$\psi^\omega$ satisfies \eqref{eq11.21}, \eqref {eq11.22}, the ball
$B(x, r)\subset\mathbb T^3$ will necessarily contain a nodal
component.

Partitioning $\mathbb T^3$ in boxes $Q_\alpha$ of size $\sim\frac
1\lambda$ and observing that $\vp^\omega$ and  any translate
$\vp^\omega (\cdot +a), a\in\mathbb T^3$, are random variables with
the same distribution, the preceding implies that, with large
probability in $\omega$, the nodal set $N_\omega$ of $\vp^\omega$
satisfies
$$
\# \{\alpha; Q_\alpha \text { contains a component of $N_\omega$}
\}\sim \lambda^3
$$
and hence $\vp^\omega$ has at least $\sim \lambda^3$ nodal
components.

It remains to justify \eqref{eq11.20}.

Take a radial bumpfunction $\eta$ on $\mathbb R^3$ such that
\be\label{eq11.23} \eta(x) \sim e^{-|x|} \text { and } \hat\eta (x)
=1 \text { for } |x|=1 \ee and set $\eta_\lambda(x) =\lambda^3
\eta(\lambda x)$. Thus $\hat\eta_\lambda (x) =1$ for $|x|=\lambda$
and therefore \be\label{eq11.24} G^\omega= G^\omega* \eta_\lambda.
\ee Since $\int\eta_\lambda=\int\eta<C$, clearly \be\label{eq11.25}
\int_{\mathbb R^3} |G^\omega(x)|\eta_\lambda (x) dx < C \ee with
probability at least $\frac 12$ in $\omega$.

Let $|y|\leq r$. By \eqref{eq11.24}, $|G^\omega(y)|\leq
\int|G^\omega(x)|\eta_\lambda (x-y) dx$, and since $\eta_\lambda
(x-y)\sim\eta_\lambda (x)$ for $|y|\lesssim \frac 1{\lambda}$ by the
choice of $\eta$ in \eqref {eq11.23}, \eqref {eq11.20} follows from
\eqref {eq11.25}.

This completes the proof of Theorem \ref{Theorem7}.
\end{proof}

\appendix
\section{ Lattice points in caps $(d\geq 4)$}\label{sec:appendix}

Let $N=R^2 \in\mathbb Z$. We show the following \be\label{app1}
|\mathcal E_R\cap C_r|\lesssim
\begin{cases} \frac{r^{d-1}}{R} + r^{d-3}  &\text { if } \quad d\geq 7\\[8pt]
\frac {r^5}R +(\log \omega(N))^2 r^3  &\text { if } \quad d=6\\[8pt]
\frac {r^4}R+ r^{2+\ve} R^\ve & \text { if } \quad d=5\\[8pt]
\frac {r^3} R (\log \omega (N))^2 + r^{\frac 32+\ve}R^\ve  &\text {
if } \quad d=4\end{cases} \ee Let $N=R^2$ and $b=(b_1, \ldots,
b_d)\in \mathcal E\cap C_r$. Then \be\label{app2}
\begin{aligned}
|\mathcal E_R\cap C_r|&\leq |\{x\in\mathbb Z^d| x_1^2+\cdots+ x_d^2 =N \text { and } |x_j-b_j|\leq r\}|\\
&\leq \Big|\Big\{y\in\mathbb Z^d\cap B_r\Big| \sum^d_{j+1} y_j^2+
2b_jy_j=0\Big\}\Big|.
\end{aligned}
\ee

Let $\gamma$ be a smooth bump function. Express \eqref{app2} by the
circle method as \be\label{app3} \int_\mathbb T \prod^d_{j=1}
\Big[\sum \gamma\Big(\frac yr\Big) e\big(( y^2+2b_jy)t\big)\Big] dt.
\ee Denote \be\label{app4} G(t, \vp)=\sum_y \gamma\Big(\frac yr\Big)
e\Big(y^2t+ y\vp\Big). \ee Let \be\label{app5} t=\frac aq+\beta, q<
r, (a, q)=1 \text { and } |\beta|<\frac 1{qr}. \ee By Poisson
summation \be\label{app6} G(t, \vp)\sim\sum_{m\in\mathbb Z} S(a, m;
q)J(\vp,\beta, m; q) \ee where \be\label{app7} S(a, m; q)= \frac 1q
\sum^{q-1}_{k=0} e_q (k^2 a-km) \ee and \be\label{app8} J(\vp,
\beta, m; q)=\int_{\mathbb R} \gamma\Big(\frac y r\Big)
e\Big(\Big(\vp+\frac mq\Big) y+ y^2\beta\Big) dy. \ee Note that
certainly \be\label{app9} |J(\vp, \beta, m; q)|\lesssim \min\Big( r,
\frac 1{\sqrt {|\beta|}}\Big) \ee and also (for appropriate choice
of $\gamma$) \be\label{app10} |J|\lesssim r \, e^{-(r|\vp+\frac
mq|)^{1/2} } \ \text { if } \ \Big|\vp+\frac mq\Big|> 2r|\beta|. \ee
In particular, it follows from \eqref{app10} that \eqref {app6} only
involves a few significant terms.

Substitution of \eqref{app6} in \eqref {app3} gives \be\label{app11}
\sum_{m_1, \ldots, m_d}\Big\{\prod^d_{j=1} S(a, m_j- 2ab_j, q)
\Big\} \Big\{ \prod^d_{j=1} J(2b_j\beta, \beta, m_j; q)\Big\} \ee
where it remains to perform the sum over $(a; q)=1$, integrate in
$|\beta|< \frac 1{rq}$ and sum over $q< r$.

Since \be\label{app12} S(a, m; q)= S(1, 0, q) \Big(\frac aq\Big) e_q
(m^2 a') \quad a'a \equiv 1 (\mod q) \ee the first factor in
\eqref{app11} equals \be\label{app13} S(1, 0, q)^d \Big(\frac
aq\Big)^d \ e_q\Big(a' \Big(\sum_j(m_j-2ab_j)^2\Big)\Big)\sim S(1,
0, q)^d \ \Big(\frac aq\Big)^d \ e_q (4aN+a'|m|^2). \ee Summing
\eqref{app13} over $a$, $(a, q)=1$ (the sum factors over the prime
factorization of $q$) and applying Weil's bound on the Kloosterman
sum ($d$ even) or Sali\'e sum ($d$ odd), gives the bound
\be\label{app14}
q^{-\frac d2}\cdot \begin{cases} \sqrt q\tau (q) \qquad &d  \ \text { odd}\\
\sqrt q \tau(q) (q, N)^{\frac 12} \qquad &d  \ \text {
even}.\end{cases} \ee Hence \be\label{app15}
\eqref{app3}\leq\sum_{q\leq r}\frac {(q, N)^{\frac 12}
\tau(q)}{q^{\frac {d-1} 2}} \sum_{m_1, \ldots, m_d}
\int\prod^d_{j=1} |J(2b_j\beta, \beta, m_j; q)| d\beta. \ee Since
$|b|=R$, we may  assume $|b_1|\sim R$.

From \eqref{app9}, \eqref{app10}) \be\label{app16} \sum_m|J(\vp,
\beta, m; q)|\lesssim \frac 1{\sqrt \beta} + re^{-(\frac rq)^{\frac
12} }\lesssim \frac 1{\sqrt \beta} \ee since $|\beta|< \frac 1{rq}$.
Hence \be\label{app17} \eqref{app15} \leq \sum_{q\leq r} \frac{(q,
N)^{\frac 12} \tau(q)}{q^{\frac {d-1}2}} \ \sum_{m_1}
\int\frac{|J(2b_1 \beta, \beta, m_1, q)|}{(\sqrt\beta+\frac
1r)^{d-1}}d\beta. \ee From \eqref{app10} \be\label{app18} |J(2b_1
\beta, \beta, m_1, q)|< r \, e^{-(r|2b_1\beta+\frac{m_1} q|)^{\frac
12}} \text { if } \Big|2b_1\beta+ \frac {m_1} q\Big|
> 2r|\beta|
\ee and hence \be\label{app19} \sum_{m_1} |J(2b_1 \beta,\beta, m_1,
q)|< r\, e^{-(\frac rq \Vert 2b_1q\beta\Vert)^{\frac 12}} \text { if
} \ \Vert 2b_1 q\beta\Vert > 2rq.|\beta|. \ee We use this property
to get a better estimate.

Write \be\label{app20} \beta =\frac \ell{2b_1q} +\beta',
|\beta'|<\frac 1{4|b_1|q} \text { and } \ell \in\mathbb Z, |\ell
|\lesssim \frac{|b_1|}r \sim \frac Rr. \ee Thus \eqref{app19}
implies \be\label{app21} \sum_m|J(2b_1\beta, \beta, m; q)| \lesssim
r \, e^{-(rR|\beta'|)^{\frac 12}} \text { if } |\beta'|> 10
\frac{r\ell}{R^{2} q}. \ee

\noindent Contribution of $|\beta|<\frac 1{r^2}$:

For such $\beta$,
 from \eqref{app20}, $|\ell|\lesssim \frac {Rq}{r^2}$ and
 \eqref{app21} will hold if $|\beta'|\gg \frac 1{Rr}$.

Since for $|\beta'|\lesssim \frac 1{Rr}$, \eqref{app21} is certainly
true, it is always valid.

The $m_1$-sum in \eqref{app17} is therefore bounded by
\be\label{app22}
\begin{aligned}
&r^{d-1} \Big(1+\frac{Rq}{r^2}\Big) r\int e^{(-rR|\beta'|)^{\frac 12}} d\beta'\\
&< r^{d-1} \Big( \frac 1R+\frac q{r^2}\Big)
\end{aligned}
\ee This gives the contribution \be\label{app23} \frac{r^{d-1}}R
\Big(\sum_{q\leq r} \frac{\tau(q) (q, N)^{\frac 12}}{q^{\frac{d-1}
2}}\Big) + r^{d-3} \Big(\sum_{q\leq r}\frac{\tau(q) (q, N)^ {\frac
12}}{ q^{\frac{d-3}2}}\Big). \ee

\noindent Contribution of $|\beta|>\frac 1{r^2}$:

Let $|\beta|\sim\frac B{r^2}$ with $B<\frac rq$. Then $|\ell|\sim
\frac {Rq B}{r^2}$ and \eqref{app21} will hold if $|\beta'|\gtrsim
\frac B{rR}$.

Using also \eqref {app16} the contribution in \eqref{app17} is at
most \be\label{app24}
\begin{aligned}
&\sum_{q\leq r} \frac {(q, N)^{\frac 12} \tau (q)}{q^{\frac
{d-1}{2}}} \Big( \frac {r^2}B\Big)^{\frac {d-1}2}
\Big(1+\frac {RqB}{r^2}\Big) \Big(\frac B{rR} \, \frac r{\sqrt B} + r\int e^{-r(R|\beta'|)^{\frac 12}} d\beta'\Big)\\
&\leq \frac{r^{d-1}}R B^{-\frac {d-2}2} \sum_{q\leq r} \frac{(q,
N)^{\frac 12} \tau(q)}{q^{\frac {d-1} 2 }} +r^{d-3} \sum_{q\leq r}
\frac {\tau(q)(q, N)^{\frac 12}} {q^{\frac{d-3}2}} B^{-\frac{d-4}2}.
\end{aligned}
\ee Summing \eqref{app24} over dyadic values of $B<\frac rq$ gives
\eqref{app23}, except if $d=4$, where in the second sum there is an
additional $\log\frac rq$ factor.

It remains to estimate the $q$-sums in \eqref{app23}
\be\label{app25} \sum_{q\leq r} \frac {(q, N)^{\frac 12}
\tau(q)}{q^{\frac {d-1}2}} \leq\Big(\sum_{\substack{ c|N\\ c\leq r}}
\frac {\tau(c)}
{c^{\frac d2-1}}\Big) \Big(\sum_{q_1< r} \frac {1}{q_1^{\frac{d-1}2 -\ve}}\Big)<\begin{cases} C \ \text { for } \ d\geq 5\\
C \, (\log\, \omega(N))^2 \ \text { for } \ d=4\end{cases} \ee and
\be\label{app26} \sum_{q\leq r} \frac {(q, N)^{\frac 12}\tau(q)}
{q^{\frac {d-3}2}} \leq \Big(\sum _{\substack { c|N\\ c\leq r}}
\frac {\tau(c)}{c^{\frac d2 -2}}\Big) \Big( \sum_{q_1< r} \frac
{\tau(q_1)}{q_1^{\frac {d-3} 2}} \Big) \ll
\begin{cases} C \ \text { for } d\geq 7\\
C\, (\log\, \omega (N))^2 \ \text { for } d=6\\
R^\ve \ \text { for } d=5.\end{cases} \ee while for $d=4$, we have
\be\label{ass27} \sum_{q\leq r} \frac {(q, N)^{\frac 12}
\tau(q)}{q^{\frac 12}} \log\frac rq \ll r^{\frac 12+\ve}R^\ve. \ee
This gives \eqref{app1}.

\end{document}